\documentclass[11pt,reqno,letterpaper]{amsart}
\usepackage[utf8]{inputenc}
\usepackage{amsfonts}
\usepackage{orcidlink}
\usepackage{graphicx}
\usepackage{verbatim}
\usepackage{amscd}
\usepackage{amsmath, physics}
\usepackage{amssymb}
\usepackage{latexsym}
\usepackage{hyperref}
\usepackage[all]{xy}
\usepackage{color}
\usepackage{mathrsfs}
\usepackage{lmodern}
\theoremstyle{definition}
\newtheorem{remark}{\bf Remark}

\newtheorem{example}{\bf Example}

\theoremstyle{plain}
\newtheorem{corollary}{\bf Corollary}
\newtheorem{lemma}{\bf Lemma}

\newtheorem{proposition}{\bf Proposition}

\newtheorem{theorem}{\bf Theorem}
{}
\numberwithin{equation}{section}
\usepackage{tikz}
\usepackage{empheq,etoolbox}
\patchcmd{\subequations}
  {\theparentequation\alph{equation}}
  {\theparentequation.\alph{equation}}
  {}{}
\hypersetup{
  colorlinks=true,		
  linkcolor=black,		
  citecolor=black,		
  filecolor=black,		
  urlcolor=black,		
 bookmarksdepth=4
}

\usepackage{setspace}
\let\oldtocsection=\tocsection
\let\oldtocsubsection=\tocsubsection
\renewcommand{\tocsection}[2]{\hspace{-1mm}\bf\oldtocsection{ #1}{#2}}
\renewcommand{\tocsubsection}[2]{\hspace{6.9mm}\oldtocsubsection{#1}{#2}}

\begin{document}
\title[Mean curvature flow in a harmonic-Ricci flow background]{Mean curvature flow into an ambient Riemannian manifold evolving by Ricci flow coupled with harmonic map heat flow}
\author[José N.V. Gomes]{José N.V. Gomes$^1$\orcidlink{0000-0001-5678-4789}}
\author[Matheus Hudson]{Matheus Hudson$^2$}
\author[Carlos Maurício]{Carlos M. de Sousa$^3$}
\address{$^{1}$Departament of Mathematics, Universidade Federal de São Carlos, Rod. Washington Luís, Km 235, 13.565-905, São Carlos, São Paulo, Brazil.}
\address{$^2$Institut für Mathematik, Carl von Ossietzky Universität Oldenburg, Ammerländer Heerstr, 114–118, 26129, Oldenburg, Germany.}
\address{$^{3}$Departament  of Mathematics, Universidade Federal de Rondônia, Av. Presidente Dutra - até 2965 - lado ímpar, 76.801-059, Porto Velho, Rondônia, Brazil.}
\email{$^1$jnvgomes@ufscar.br}
\email{$^{2}$matheus.hudson.gama.dos.santos@uol.de}
\email{$^3$carlosmauricio@unir.br}
\urladdr{$^{1,2}$https://www2.ufscar.br}
\urladdr{$^2$https://uol.de}
\urladdr{$^{3}$https://unir.br/homepage}
\keywords{Ricci harmonic flow, mean curvature flow.}
\subjclass[2020]{Primary: 53E10; Secondary: 53E20, 58E20}

\dedicatory{Dedicated to Professor Keti Tenenblat on the occasion of her 80th birthday}

\begin{abstract}
The main objective of this article is to study the mean curvature flow into an ambient compact smooth manifold M with boundary and with a Riemannian metric that evolves by a self-similar solution of the Ricci flow coupled with the harmonic map heat flow of a map from M to a Riemannian manifold N. In this context, we address a functional associated with this flow and calculate its variation along parameters that preserve the weighted volume measure. An extension of Hamilton’s differential Harnack expression appears by considering the boundary of M evolving by mean curvature flow, which must vanish on the gradient steady soliton case. Next, we obtain a Huisken monotonicity-type formula for the mean curvature flow in the proposed background. We also show how to construct a family of mean curvature solitons and establish a characterization of such a family.
\end{abstract}

\maketitle

\section{\bf Introduction}
This article concerns mean curvature flow in Riemannian manifolds evolving by Ricci flow coupled with harmonic map heat flow. Many works inspire our approach, as we shall describe now.

We begin with the work by Eells and Sampson~\cite{Eells_Sampson}, which pioneered the study of harmonic maps that arise from the variation of the energy-type functional as a generalization of Dirichlet's energy functional. There, they aimed to establish the existence of harmonic maps which are homotopic to a given map $\phi: (M,g) \rightarrow (N,\gamma)$, where $(M,g)$ and $(N,\gamma)$ are closed Riemannian manifolds, i.e., compact and without boundary. For it, they considered the \emph{energy functional} $E$ of $\phi$ as follows
\begin{equation*}
E(\phi):=\frac{1}{2}\int_{M}|\nabla\phi|^{2}dM,
\end{equation*}
and they showed that, for a smooth family of maps $\phi_t: (M,g) \rightarrow (N,\gamma)$, with $t \in (-\epsilon,\epsilon)$, variational vector field $V$ and $\phi_0=\phi$, the first variation formula of $E$ is given by 
\[ \dfrac{d}{dt}\Big|_{t=0}E(\phi_t)= - \int_{M}\langle V, \tau_{g,\gamma}\phi \rangle dM,\]
where $\tau_{g,\gamma}\phi$ denotes the tension field of $\phi$, which depend on the Riemannian metrics $g$ and $\gamma$.  In particular, a harmonic map $\phi$ (i.e., $\tau_{g, \gamma}\phi=0$) is a critical point of $E$, see Section~\ref{Sec:preliminaries} for details.

The idea is to deform a given map $\phi \in C^{\infty}(M, N)$ along the flow given by $\tau_{g,\gamma}\phi_t$ to obtain a harmonic map free-homotopic to $\phi.$  When such a deformation is possible, its flow $\phi_t$ becomes a solution of the system of parabolic partial differential equations
\begin{equation}\label{Eells_Sampson_Problem}
\dfrac{\partial\phi_t}{\partial t} =\tau_{g,\gamma} \phi_t \quad \hbox{with} \quad \phi_t|_{t=0}=\phi.
\end{equation}

System~\eqref{Eells_Sampson_Problem} is known as the harmonic map heat flow. For our context, we highlight the following particular result by Eells and Sampson. We observe that they worked in a more general setting by imposing some boundedness on the embedding of $N$ in some Euclidean space $\mathbb{R}^d$; such conditions are automatically fulfilled if $N$ is compact. Eells and Sampson's theorem reads as follows. 

\vspace{0.2cm}
\emph{Let $(M,g)$ and $(N,\gamma)$ be closed Riemannian manifolds, and consider a smooth map $\phi:(M,g) \to (N,\gamma)$. If $(N,\gamma)$ has nonpositive Riemannian curvature, then there exists a unique global smooth solution of \eqref{Eells_Sampson_Problem} which converges smoothly to a harmonic map homotopic to $\phi$.}

\vspace{0.2cm}
One of the important aspects of harmonic maps is that they generalize the concept of harmonic functions. In particular, closed geodesics and minimal surfaces are some examples. If $\phi$ is an isometric immersion of a Riemannian manifold $M$ in an Euclidean space, then the tension field has the simplified notation $\Delta_g\phi$ and coincides with the mean curvature $H(\phi)$ (see Eells and Sampson~\cite{Eells_Sampson}, and Takahashi~\cite{Takahashi}). Hence, $\phi$ is harmonic if and only if it is minimal. Moreover, any isometry of $M$ is harmonic, and any covering map is harmonic.

Hamilton extended Eells and Sampson's theorem for compact Riemannian manifolds with boundary. He showed that the first variation formula of the energy functional $E(\phi)$ of a smooth map $\phi: (M,g) \rightarrow (N,\gamma)$, now between Riemannian manifolds with boundary, is given by
\begin{equation*}
\left.\begin{matrix}
\dfrac{d}{dt}
\end{matrix}\right|_{t=0}E(\phi_t)=-\int_{M}\langle V, \tau_{g, \gamma}\phi \rangle dM + \int_{\partial M}\langle V, \nabla_{0}\phi \rangle dA.
\end{equation*}
So, a harmonic map $\phi$ with Neumann boundary condition $\nabla_{0}\phi = 0$ is a critical point of the energy functional. Hamilton noted that there are three natural boundary value problems to be addressed: $(i)$ Dirichlet Problem for a harmonic map $\phi$ with given values $\phi=\hat\phi$ on $\partial M$; $(ii)$ Neumann Problem for a map $\phi$ not specified on $\partial M$ but with auxiliary condition that the normal derivative $\nabla_{0}\phi=0$ on $\partial M$; and $(iii)$ Mixed Problem, which, in contrast with the two previous cases, this one considers $\partial N$, since it is assumed that $\phi$ maps $\partial M$ into $\partial N$, but in an arbitrary form, and also that the normal derivative $\nabla_{0}\phi$ taken at a point in $\partial M$ is normal to $\partial N$. In all cases, he proved existence results by assuming that $N$ has nonpositive Riemannian curvature, moreover, that $\partial N$ is convex, for $(i)$ and $(ii)$, and totally geodesic for $(iii)$. For details, see Hamilton~\cite{Hamilton75}.

Motivated by the previously discussed theory coupled with the promising case of Ricci flow introduced by Hamilton~\cite{hamilton1982three}, we consider a family of closed hypersurfaces $\Sigma_t$ in $(M,g(t))$  and a family of smooth maps $\phi_t: (M,g(t))\to (N,\gamma)$ with Riemannian metrics $g(t)$ evolving by some geometric flow and $\Sigma_t$ evolving by mean curvature flow. We now contextualize with results and historical data our setting of study. 

It is known that the Ricci flow was expected to have a gradient-like structure, as well as the mean curvature flow case. Indeed, this was one of Perelman's contributions by modifying the Hilbert-Einstein functional in the context of weighted compact smooth manifolds. He defined the functional ${\mathcal F}(g,f)$ on the space of metrics and smooth functions on a closed smooth manifold, whose variation $\delta{\mathcal F}(g,f)$ provides a gradient-like structure to the Ricci flow with weighted measure-preserving, see Perelman~\cite{Perelman}. 

Four years later, List~\cite{Bernhard_List} presented a connection between Ricci flow on an $m$-dimensional closed Riemannian manifold $(M,g)$ and the Einstein's static vacuum equations through a coupled system of Ricci flow and heat equation with a coupling constant $\alpha_{m}=(m-1)/(m-2)$, with $m\geqslant3$, and then he defined a functional $\mathcal F(g,f,w)$ on the space of metrics and cartesian product of smooth functions on a closed smooth manifold, whose variation $\delta{\mathcal F}(g,f, w)$ provides a gradient-like flow from this coupled system.

In the boundary case, Ecker~\cite{Klaus-Ecker} defined a version of the $\mathcal{W}$-functional of Perelman for the Ricci flow on bounded Euclidean domains with smooth boundary. Curiously, Hamilton's differential Harnack expression~\cite{Hamilton} on the boundary integrand appears in its time-derivative formula. Based on Ecker's work, Lott \cite{John_Lott} defined the functional $I_{\infty}(g,f)$ on the space of metrics and smooth functions on a compact smooth manifold with boundary to be a weighted version of the Gibbons-Hawking-York action~\cite{Gibbons-Hawking,York} from which he found an extension of Hamilton's differential Harnack expression on the boundary integrad. It is also worth noting that Magni, Mantegazza and Tsatis~\cite{magni2013flow} found a Huisken monotonicity-type formula~\cite{Huisken1990} for the mean curvature flow in an ambient smooth manifold with Riemannian metric that evolves by a self-similar solution to the Ricci flow. 

More recently, the first and second authors considered Lott’s program in the context of mean curvature flow in an extended Ricci flow background. They studied variational properties of an appropriate extended version of Lott's functional in the context of List's work, namely, the extended weighted Gibbons-Hawking-York action $I^{\alpha_m}_\infty(g, f, w)$  on an $m$-dimensional compact smooth manifold with boundary. They obtained evolution equations for the second fundamental form and the mean curvature in an extended Ricci flow background, and then an extension of Hamilton's differential Harnack expression appears as well as a Huisken monotonicity-type formula for the mean curvature solitons in this background, see Gomes and Hudson~\cite{Gomes_Hudson} for details. 

In the more general context, Müller~\cite{MullerThesis, muller2012ricci} worked in a new geometric flow which consists of a coupled system of the Ricci flow on a closed Riemannian manifold $(M, g)$ with the harmonic map heat flow of a map $\phi: (M,g) \rightarrow (N,\gamma),$ where $(N, \gamma)$ is a closed Riemannian manifold. Precisely, he considered a family of Riemannian metrics $g(t)$ on $M$ and a family of smooth maps $\phi(t)$ from $M$ to $N$ to define $(g(t), \phi(t))_{t \in[0, T)}$ as a solution to the Ricci flow coupled with harmonic map heat flow, \textit{$(R H)_{\alpha}$ flow} for short, namely
\begin{align}\label{(RH)}
 \left\{\begin{array}{lcl}
\dfrac{\partial}{\partial t}g(t)=-2 \operatorname{Ric}_{g(t)} +2\alpha(t) \nabla \phi(t) \otimes \nabla \phi(t),\\[1ex]
\dfrac{\partial}{\partial t} \phi(t)=\tau_{g(t),\gamma} \phi(t),
\end{array}\right.   
\end{align}
where $\alpha(t)$ is a nonnegative coupling constant. For an account of $(RH)_{\alpha}$ flow, including proof of short-time existence and uniqueness of solutions to~\eqref{(RH)}, see~\cite[Sect.~4.2]{muller2012ricci}.

Müller realized that his coupled system may behave less singularly than the Ricci flow or the standard harmonic map flow alone. To interpret~\eqref{(RH)} as a gradient flow by means a functional ${\mathcal F}_{\alpha}(g, f, \phi)$ for a fixed measure, he worked with the heat operator $\square=\frac{\partial}{\partial t} - \Delta_g$ whose formal adjoint $\square^{*}$ is given by
\begin{align}\label{adjoint:square}
\square^{*}=-\frac{\partial}{\partial t} - \Delta_g +R_g-\alpha|\nabla\phi|^{2}    
\end{align}
along the $(RH)_\alpha$ flow. 

Müller's approach motivated the first theorem of this article. Next, we continue to establish our study context more precisely. 

A gradient soliton to the $(RH)_{\alpha}$ flow is a self-similar solution $\big(\overline g(t), \overline\phi(t) \big)$ of ~\eqref{(RH)} given by
\begin{align*}
\left\{
\begin{array}{lcl}
\overline{g}(t)  = \sigma(t)\psi_t^* g,\\ [1ex]
\overline{\phi}(t) = \psi_t^* \phi,
\end{array}
\right.
\end{align*}
for some initial value $(g,\phi)$, where $\psi_t$ is a  smooth one-parameter family of diffeomorphisms of $M$ generated from the flow of $\nabla_g f/\sigma(t)$, $f\in C^\infty(M)$, and  $\sigma(t)$ is a positive smooth function on $t.$ 
By setting $\overline{f}(t)=\psi_t^*f$,  from~\eqref{(RH)} we can obtain
\begin{align}
\left\{
\begin{array}{rcl}\label{mod_grad_Ricci_soliton}
\operatorname{Ric}_{\overline{g}} + \nabla_{\overline{g}}^2\,\overline{f} - \alpha  \nabla\overline{\phi} \otimes \nabla\overline{\phi} &=& \dfrac{c}{2(T-t)}\overline{g},\\[2ex]
\tau_{\overline{g},\gamma} \overline{\phi} &=& \langle\nabla_{\overline{g}} \overline{\phi},\nabla_{\overline{g}} \overline{f}\rangle,
\end{array}
\right.
\end{align}
where $c= 0$ in the steady case (for $t \in  \mathbb{R}$ and $\psi_0=\operatorname{Id}$), $c= 1$ in the shrinking case (for $t \in (-\infty, T)$ and $\psi_{T- 1}=\operatorname{Id}$) and $c=1$ in the expanding case (for $t \in (T,\infty)$ and $\psi_{T+1}=\operatorname{Id}$).  Moreover,
\begin{align}\label{self-solution}
\dfrac{\partial }{\partial t} \overline{f} = |\nabla_{\overline{g}} \overline{f}|_{\overline{g}}^2\,.
\end{align}
Function $\overline{f}$ is called the \emph{potential function}. 

As in \cite{Gomes_Hudson}, we consider the mean curvature flow in the following context: let $(g(t),\phi(t))$ be an $(R H)_{\alpha}$ flow in $M \times [0,T)$. Given an $(m-1)$-dimensional closed smooth manifold $\Sigma$, and let $\{x(\cdot,t); t \in [0,T)\}$ be a smooth one-parameter family of immersions of $\Sigma$ in $M$.  For each $t\in [0,T),$ set  $x_t := x(\cdot, t)$ and $\Sigma_t$ for the hypersurface $x_t(\Sigma)$ of $(M,g(t)),$ which we can also write as $\Sigma_t:=\big(\Sigma, x_t^*g(t)\big)$. Suppose  that the family  $\mathscr F :=\{\Sigma_t\,;\, t\in [0,T)\}$ evolves under mean curvature flow, MCF for short, \begin{align*}
\left\{
\begin{array}{rcl}
\dfrac{\partial}{\partial t}x(p,t) &=&  H(p,t) e(p,t),\\ [1ex]
x(p, 0) &=& x_0(p),
\end{array}
\right.
\end{align*}
where $H(p,t)$ and $e(p,t)$ are the mean curvature and the unit normal of $\Sigma_t$ at $p\in\Sigma$, respectively. In this setting, we say that $\mathscr{F}$ is a \emph{MCF in the $(g(t),\phi(t))-(RH)_\alpha$ flow background}. In the particular case $\big(g(t),\phi(t)\big)=\big(\overline{g}(t),\overline{\phi}(t)\big)$ is a self-similar solution to the $(RH)_\alpha$ flow on $M$ with potential function $\overline{f}$, a hypersurface $\Sigma_t\in \mathscr F$ is a \emph{mean curvature soliton}, if
\begin{align*}
H(p,t)+e(p,t)\overline{f}=0 \quad\hbox{on}\quad\Sigma.
\end{align*}
Here, $e(\,\cdot\,,t)$ must be the inward unit normal vector field on $\Sigma_t.$ 

Now, we consider an $m$-dimensional compact smooth manifold $M$ with boundary $\partial M$. Let ${\rm met}(M)$ be the set of all Riemannian metrics $g$ on $M.$ We define the  functional $\mathcal F^{\alpha}_\infty$  on the product $\mathscr P(M,N):={\rm met}(M)\times C^\infty (M)\times C^\infty(M,N)$ as
\begin{align}\label{MGHYF}
\mathcal F^{\alpha}_\infty(g, f, \phi):=\int_{M} \Big(R_\infty -\alpha |\nabla \phi|^2 \Big)e^{-f} dM + 2\int_{\partial M}H_\infty e^{-f}dA,
\end{align}
where $R_\infty:= R_g + 2 \Delta_g f - |\nabla f|_g^2$ is the  \emph{weighted scalar curvature} of $g$, the function $H_\infty:=H_g + e_0 f$ is  the \emph{weighted mean curvature} with respect to the inward unit normal vector field $e_0$ on $\partial M$, and the forms $dM$ and $dA$ are the $m$-dimensional Riemannian measure of $(M,g)$ and the $(m-1)$-dimensional Riemannian measure of $(\partial M,g),$ respectively.

We observe that $\mathcal F^{\alpha}_\infty$ is the proper extension for our context of the energy functionals  
$E(\phi),$ $\mathcal F(g,f),$ $\mathcal F(g,f,w),$ $I_\infty(g,f),$  $I^{\alpha_m}_\infty(g, f, w)$ and  $\mathcal F_{\alpha}(g,f,\phi)$ previously mentioned. Furthermore, it is already clear that $R_\infty$ arises quite naturally, as observed by Perelman~\cite[Sect.~1.3]{Perelman}, and $H_\infty$ is in fact the appropriate geometric object when we are using a weighted measure (see, e.g., Gromov~\cite[Sect.~9.4.E]{Gromov}).

Our first main result is a variational formula for  $\mathcal F_{\infty}^\alpha$ from which we can obtain a gradient-like structure for $(RH)_\alpha$ flow and an extension of Hamilton’s differential Harnack expression of the mean curvature flow in Euclidean space. It reads as follows (see Sections~\ref{Sec:preliminaries} and \ref{Sec:VWEGHYA} for definitions and notations).

\begin{theorem}\label{principal_theorem}
Let $M$ be an $m$-dimensional compact smooth manifold with boundary $\partial M,$ and let $\mathscr F$ be the MCF of $\partial M$ in the $\big(g(t),\phi(t)\big)-(RH)_\alpha$ flow background with  Neumann boundary condition $\nabla_0 \phi =0.$ If $u:=e^{-f}$ is a solution to the conjugate heat equation
\begin{equation}\label{back:heat:eq}
\square^{*}u  =0 \quad  \mbox{in} \quad M\times[0,T)
\end{equation}
with $e_0 u = Hu$ on $\partial M$, then
\begin{align*}
\dfrac{d}{d t}\mathcal F^{\alpha}_\infty =& 2 \int_{M} \Big(|{\operatorname {Ric}}\, + \nabla^2 f -\alpha \nabla\phi \otimes \nabla\phi|^2 + \alpha|\tau_{g,\gamma} \phi - \langle\nabla \phi, \nabla f\rangle|^2\Big) e^{-f}dM \\
& + 2 \int_{\partial M} \Big(\frac{\partial H}{\partial t}  - 2\langle \widehat{\nabla} f, \widehat{\nabla} H\rangle + \mathcal{A}
(\widehat{\nabla} f, \widehat{\nabla} f)  + 2 R^{0 i}\widehat{\nabla}_i f  - \dfrac{1}{2} \nabla_0 R\\
&-HR_{00} +  \alpha    \mathcal{A}( \widehat{\nabla} \phi, \widehat{\nabla} \phi) \Big)  e^{-f} dA,
\end{align*}
where $\mathcal{A}$ is the second fundamental form of $\partial M,$ and $\widehat{\nabla}$ denotes the gradient on $\partial M.$ 
\end{theorem}

For the proof of Theorem~\ref{principal_theorem}, we first study the MCF in an extended Ricci flow background, and then ``translate'' the results for the context of the $(RH)_{\alpha}$ flow. We also obtain an extension of Hamilton’s differential Harnack expression from the mean curvature flow in Euclidean space, but now, to the more general context of MCF in the $(RH)_{\alpha}$ flow background, which must vanish on the gradient steady soliton to this flow, see Corollary~\ref{type-Harnack-background}. 

Our second main result is a Huisken monotonicity-type formula for the MCF in the $(RH)_{\alpha}$ flow background. 

\begin{theorem}\label{Huisken_monotonicity}
Let $(M,g)$ be an $m$-dimensional Riemannian manifold, and let $\Sigma$ be an $(m-1)$-dimensional closed smooth manifold. Consider $\mathscr F$ the MCF of $\Sigma$  in the $(\overline g,\overline \phi)-(RH)_{\alpha}$ flow background with potential function $\overline{f}$. Denote by $dA_{\overline{g}}$ the $(m-1)$-dimensional Riemannian measure on $\Sigma$ and set $\operatorname{Area}_{\overline{f}}(\Sigma_t):=\int_{\Sigma} e^{-\overline{f}}dA_{\overline{g}}$. Under these conditions, the function $\Phi(t)$ given by:
\begin{itemize}
\item[(i)] $\mathbb{R}\ni t\mapsto\operatorname{Area}_{\overline{f}}(\Sigma_t)$ in the steady case,
\item[(ii)] $(-\infty,T)\ni t\mapsto [4 \pi(T-t)]^{-(m - 1 ) / 2}\operatorname{Area}_{\overline{f}}(\Sigma_t)$ in the shrinking case, and
\item[(iii)] $(T,\infty)\ni t\mapsto [4\pi(t-T)]^{-(m - 1 ) / 2}\operatorname{Area}_{\overline{f}}(\Sigma_t)$ in the expanding case,
\end{itemize}
is non-increasing. Moreover, $\Phi(t)$ is constant if and only if $\mathscr F$ is a family of mean curvature solitons.   
\end{theorem}

In Section~\ref{sect:grad:sol}, we address the construction of a family $\mathscr G$ of mean curvature solitons in the $(RH)_\alpha$ flow background, and we establish a characterization of such a family, as follows.
\begin{theorem}\label{NazasCharacterization}
If $\Sigma$ is an $f$-minimal hypersurface of a Riemannian manifold $(M,g)$, then $\mathscr G$ is a family of mean curvature solitons in the $(\overline g, \overline \phi)-(RH)_{\alpha}$ flow background on $M$. Moreover, any family $\mathscr F$ of mean curvature solitons in the $(\overline g, \overline \phi)-(RH)_{\alpha}$ flow background on $M$ is given by $\mathscr G$ up to reparametrization.
\end{theorem}

This paper is structured as follows. We begin in Section~\ref{Sec:preliminaries} with some definitions and basic concepts about Riemannian geometry and maps between Riemannian manifolds, by defining and commenting upon the concepts required to lay the groundwork for our proofs. In Section~\ref{Sec:VWEGHYA}, we obtain the variational formula for the functional $\mathcal F^{\alpha}_{\infty}$ under weighted measure-preserving, and characterize its critical points. In Section~\ref{MRFCHMHF} we work on the modified \((RH)_\alpha\) flow setting, as a tool to study MCF in the $(RH)_\alpha$ flow background, which is the main research object of this article. In Section~\ref{se:Ricci}, we give the proof of Theorems~\ref{principal_theorem} and \ref{Huisken_monotonicity}. In Section~\ref{Sec-EHDHE}, we provide an extension of Hamilton’s differential Harnack expression for mean curvature flow in Euclidean space to the more general context of mean curvature flow in the $(RH)_{\alpha}$ flow background. In Section~\ref{sect:grad:sol}, we give the proof of Theorem~\ref{NazasCharacterization}. In Section~\ref{section:how:to:construct}, we show how to construct a family of mean curvature solitons for the MCF in a self-similar solution to the $(RH)_\alpha$ flow by means of radial smooth functions on Euclidean space.

\section{\bf Preliminaries} \label{Sec:preliminaries}

Throughout this text, all manifolds are assumed to be orientable and connected. Consider a smooth map  $\phi: (M^m,g) \rightarrow (N^n,\gamma)$  between Riemannian manifolds $\left(M^{m}, g\right)$ and $\left(N^{n}, \gamma\right)$ with boundaries $\partial M$ and $\partial N$, respectively. We shall denote the local coordinates at $p\in M$ by $\{x^{i}\}$, the local coordinate basis by $\{\partial_{i}\}$ and the local dual coordinate basis by $\{dx^{i}\}$. Near $\partial M$, we take $x^0$ to be a local defining function for $\partial M$. We denote the local coordinates for $\partial M$ by $\{x^{\hat{i}}\}.$ We choose these coordinates near a point at $\partial M$ so that $\partial_0 \big|_{\partial M}$ coincides with the inward-pointing unit normal field $e_0$ along the boundary, moreover, we can assume that $\partial_i|_{\partial M}$ coincides with $\partial_{\hat{i}}$ along the boundary.  For $N$ we shall denote $\{y^\alpha\}$ the local coordinate at $\phi(p)$, the local coordinate basis by $\{\partial_{\alpha}\}$ and $\phi^{\alpha}:=y^{\alpha}\circ\phi.$ We shall use the convention that repeated Latin indices are summed over from $0$ to $m-1$ and repeated Greek indices are summed over from $0$ to $n-1$. In general, we are using the Einstein convention of summing over repeated indices. In dealing with flows, we shall usually simplify the notation by suppressing the parameter~$t.$ 

The metric on \( M \) is denoted by \(g=\langle , \rangle\) and \( \langle\partial_i,\partial_j\rangle=g_{ij}\), and its inverse is denoted by \( g^{ij} \) so that \( g_{ij} g^{jk} = \delta_i^k \). The forms $dM$ and $dA$ are the $m$-dimensional Riemannian measure of $(M,g),$ and the $(m-1)$-dimensional Riemannian measure of $(\partial M,g),$ respectively. We also use the classical notation $h^{ij} = g^{ik} g^{jl} h_{k l}$, for any $2$-tensor field $h$ on $M$.

We denote the Levi-Civita connection on \(TM\) by \(\nabla\) and on \(T\partial M\) by \(\widehat{\nabla}\). By simplicity, we also denote \(\nabla_i:=\nabla_{\partial_i}\),  $\nabla^i:=g^{ij}\nabla_j$ and $X^i=g^{ij}X_j$, where $X_j=\langle X,\partial_j\rangle.$

In what concerns $\partial M,$ we write $\mathcal{A}_{\hat{i}\hat{j}}:=\langle \nabla_{\partial_{\hat{i}}} \partial_{\hat{j}},e_0\rangle$ for its second fundamental form, and $H:=g^{\hat{i}\hat{j}}\mathcal{A}_{\hat{i}\hat{j}}$ for its mean curvature. Hence,
$$\mathcal{A}^{\hat{i}\hat{j}} = g^{\hat{i}\hat{k}} g^{\hat{j}\hat{l}} {\mathcal{A}}_{\hat{k}\hat{l}} \quad\text{and}\quad \mathcal{A}^{\hat{k}}\!_{\hat{i}}=g^{\hat{k}\hat{l}}\mathcal{A}_{\hat{l}\hat{i}}.$$

For all $X,Y \in \Gamma(TM)$ and $\omega \in \Gamma(T^{*}M)$, we have
$$\nabla^{T^{*}M}_{X}\omega(Y)=X\big(\omega(Y)\big)-\omega\big(\nabla_X Y\big).$$

The smooth map $\phi$ induces the fiber bundle $\phi^{*}TN$ over $M$ as follows 
$$\phi^{*}TN=\left\{(p,u); p\in M, u \in T_{\phi(p)}N\right\}=\bigcup_{ p \in M}\left\{p\right\}\times T_{\phi(p)}N.$$ 

The Levi-Civita covariant derivative $\nabla^{TN}$ of the metric $\gamma$ on $N$ induces the following covariant derivative on $\phi^{*}TN$,
$$\nabla^{\phi^{*}TN}_{X}U:=\nabla^{TN}_{\phi_{*}X}U,$$
for all $X \in \Gamma(TM)$ and $U \in \Gamma(TN)$.

The Riemannian curvature tensor is defined as
\[
\operatorname{Rm}(X, Y) Z = \nabla_Y \nabla_X Z - \nabla_X \nabla_Y Z + \nabla_{[X, Y]} Z = R^l_{
ijk}Y^jX^iZ^k\partial_l.
\]
where
\begin{align*}
R^l\!_{ijk}\partial_l&=R(\partial_i,\partial_j)\partial_k= \nabla_j \nabla_i \partial_k -
\nabla_i \nabla_j \partial_k,
\\
R^l_{ijk} &=  \partial_j \Gamma^l_{ik}  -
\partial_i \Gamma^l_{jk} + \Gamma^m_{ik} \Gamma^l_{jm} - \Gamma^m_{jk} \Gamma^l_{im} ,
\\
\Gamma^k_{ij} &= \frac{1}{2} g^{kl} (\partial_i g_{jl} + \partial_j g_{il} - \partial_l g_{ij}).    
\end{align*}
When lowering the index to the fourth position, we obtain
\[
R_{ijkl} = g_{ml} R^m_{ijk},
\]
so that $R^s\!_{ijk} = g^{ls} R_{ ijkl},$ where $R_{ijkl} = \langle R(\partial_i,\partial_j)\partial_k, \partial_l\rangle$. The Ricci tensor \( R_{ij} \) is defined as \( R_{ij} = g^{kl} R_{ikjl} \), and the scalar curvature is its trace \( R = g^{ik} R_{ik} = g^{ik} g^{jl} R_{ijkl} \). 

Thus, for a vector field \(X\), one has
\[
[\nabla_i, \nabla_j] X^k = \nabla_i \nabla_j X^k - \nabla_j \nabla_i X^k = -R^k_{ijm} X^m = g^{kl} R_{ijlm} X^m.
\]
Taking the trace in the second Bianchi identity
\[
\nabla_i R_{jklm} + \nabla_j R_{kilm} + \nabla_k R_{ijlm} = 0
\]
we obtain
\[
g^{im}\nabla_i R_{jklm}= -g^{im} \nabla_j R_{kilm} -g^{im} \nabla_k R_{ijlm} =  -\nabla_j R_{kl} + \nabla_k R_{jl}.
\]
We now trace with \(g^{jl}\) to get
\[
g^{im}\nabla_i R_{km}= -g^{jl}\nabla_j R_{kl} + \nabla_k R,
\]
which immediately implies
\begin{equation*}
\nabla^l R_{kl}:=g^{jl}\nabla_j R_{kl} = \frac{1}{2} \nabla_k R.    
\end{equation*}
Now, we compute
\begin{equation*}
\nabla\partial_i(\partial_s):=\nabla_{\partial_s}\partial_i:=\Gamma_{si}^k\partial_k=\Gamma_{ji}^k dx^j(\partial_s)\partial_k=:\Gamma_{ij}^k dx^j\otimes\partial_k(\partial_s)  
\end{equation*}
and
\begin{align*}
\nabla dx^{i}(\partial_r,\partial_s):=\nabla_{\partial_r}^{T^{*}M}\big(dx^i(\partial_s)\big)-dx^i\big(\nabla_{\partial_r}\partial_s\big)
=-\Gamma_{jk}^{i} dx^{j}(\partial_r)dx^{k}(\partial_s).
\end{align*}
In short,
\begin{equation}\label{covariant_derivative1}
\nabla\partial_i=\Gamma_{ij}^k dx^j\otimes\partial_k \quad \mbox{and} \quad \nabla dx^i = - \Gamma_{jk}^i dx^j\otimes dx^k.
\end{equation} 
Moreover, by straightforward computation
\begin{align*}
\left(\nabla\partial_{\lambda}|_{\phi}\right)(\partial_i):=\nabla_{\partial_i}^{\phi^{*}TN}\partial_{\lambda}|_{\phi}&:=\nabla_{\phi_{*}\partial_i}\partial_{\lambda}|_{\phi}=\Gamma_{\alpha\lambda}^{\beta}\circ\phi)\nabla_i\phi^{\alpha}\partial_{\beta}|_{\phi}
\end{align*}
and 
\begin{align*}
\nabla dy^{\lambda}(\partial_{i},\partial_{\alpha}|_{\phi}):= \left(\nabla^{T^*M}_{\partial_i}dy^{\lambda}\right)(\partial_{\alpha}|_{\phi})&=\nabla_{\partial_i}dy^{\lambda}(\partial_{\alpha}|_{\phi})-dy^{\lambda}(\nabla_{\partial_i}\partial_{\alpha}|_{\phi})\\
&=-dy^{\lambda}\left( (\Gamma_{\alpha\beta}^{\theta}\circ\phi)\nabla_{i}\phi^{\beta}\partial_{\theta}|_{\phi}\right)\\
&=-(\Gamma_{\alpha\beta}^{\lambda}\circ\phi)\nabla_{i}\phi^{\beta}.
\end{align*}
In short,
\begin{equation}\label{covariant_derivative2}
\nabla\partial_{\lambda}|_{\phi}\!=\!(\Gamma_{\alpha\lambda}^{\beta}\circ\phi)\nabla_i\phi^{\alpha}dx^{i}\otimes\partial_{\beta}|_{\phi}\;\hbox{and}\;
\nabla dy^{\lambda} \!=\! -(\Gamma_{\alpha\beta}^{\lambda}\circ\phi)\nabla_{i}\phi^{\beta}dx^{i}\otimes dy^{\alpha}.
\end{equation}

For a smooth function $f:M\rightarrow\mathbb{R}$, we write its gradient as $\nabla f = \nabla^i f \partial_i$ so that
$\nabla^i f = g^{ij} \nabla_j f$ and $|\nabla f|^2 
= g^{ij} \nabla_i f \nabla_j f$, where $\nabla_j f 
= \langle \nabla f, \partial_j\rangle$, and the Hessian of $f$ is given by $\nabla_k \nabla_l f$.
Moreover, we have the following expressions $\nabla^i \phi^{\lambda}= g^{ij}\nabla_j \phi^{\lambda} = g^{ij}\langle\nabla \phi^{\lambda}, \partial_j\rangle$ and  $\nabla^i \nabla^j f = g^{ik}g^{jl} \nabla_k \nabla_l f$.

We recall that the derivative $\nabla\phi$ maps linearly sections of \( TM \) to sections of \( TN \) along \( \phi \), i.e., in terms of the bundle \( \phi^* TN \), we can interpret \( \nabla\phi \) as a section of the vector bundle of homomorphisms $\operatorname{Hom}(TM; \phi^* TN).$  Furthermore, since this latter bundle is isomorphic to the induced bundle \( T^*M \otimes \phi^* TN \), we can introduce a connection $ \nabla$ on $ \Gamma(T^*M \otimes \phi^* TN)$ to obtain the second derivative \( \nabla\nabla \phi \) as the derivative of $ \nabla \phi$ concerning the connection on $\Gamma(T^*M \otimes \phi^* TN)$, thus, it is a section of the bundle $T^*M\otimes T^*M\otimes \phi^* TN.$ The \emph{tension field} $ \tau_{g,\gamma} \phi$ (or Laplacian $\Delta \phi$) is the trace of \( \nabla\nabla \phi \) with respect to the inner product on \( TM \). This defines \( \tau_{g,\gamma}\phi \) as a section of the bundle \( \phi^* TN \). Precisely,
\begin{align*}
\nabla\phi:TM&\longrightarrow \phi^{*}TN\\ 
X& \longmapsto d\phi(X),
\end{align*}
where
\begin{align*}
d\phi(\partial_j)=d(y^{\lambda}\circ\phi)(\partial_j)\partial_{\lambda}|_{\phi}=d\phi^{\lambda}(\partial_j)\partial_{\lambda}|_{\phi}=\langle\nabla\phi^{\lambda},\partial_j\rangle_{_M}\partial_{\lambda}|_{\phi}=\nabla_j\phi^{\lambda}\partial_{\lambda}|_{\phi}
\end{align*}
and
\begin{align*}
d\phi(X)=\nabla_j\phi^{\lambda}dx^{j}(X)\partial_{\lambda}|_{\phi}.
\end{align*}
Therefore,
\begin{equation}\label{Oldenburg1}
\nabla\phi=\partial_j\phi^{\lambda}dx^{j}\otimes\partial_{\lambda}|_{\phi}=\nabla_j\phi^{\lambda}dx^{j}\otimes\partial_{\lambda}|_{\phi}.
\end{equation}
Taking $X=\nabla f$ and writing $\nabla f = g^{ik}\nabla_i f\partial_k,$ we get
\begin{align*}
\langle\nabla\phi,\nabla f\rangle:=\nabla\phi(\nabla f)=g^{ik}\nabla_j\phi^{\lambda}\nabla_i f dx^{j}(\partial_k)\partial_{\lambda}|_{\phi}=\langle \nabla f , \nabla \phi^{\lambda}\rangle\partial_{\lambda}|_{\phi}.
\end{align*}
Taking $X=e_0$, we have
\begin{align*}
\nabla_0\phi:=\nabla\phi(e_0)=\nabla_j\phi^{\lambda} dx^{j}(e_0)\partial_{\lambda}|_{\phi}=e_0\phi^{\lambda}\partial_{\lambda}|_{\phi}.
\end{align*}
By using \eqref{covariant_derivative1} and \eqref{covariant_derivative2}, one has
\begin{eqnarray}\label{coordinate_of_tensionfield}
\big(\nabla_{\partial_i}\nabla\phi\big)(\partial_j,dy^\lambda)&=&\partial_i\big(\nabla\phi(\partial_{j},dy^{\lambda})\big)-\nabla\phi(\nabla_{\partial_i}\partial_{j},dy^{\lambda})-\nabla\phi(\partial_{j},\nabla_{\partial_i}dy^{\lambda})\nonumber\\
&=& \partial_{i}\partial_{l}\phi^{\theta}dx^{l}(\partial_{j})\partial_{\theta}(dy^\lambda)-\Gamma_{ij}^{k}\nabla_{l}\phi^{\theta}dx^{l}(\partial_{k})\partial_{\theta}(dy^\lambda)\nonumber\\
&&+\big(\Gamma_{\alpha\beta}^{\lambda}\circ\phi\big)\nabla_{i}\phi^{\beta}\nabla_{l}\phi^{\theta}dx^{l}(\partial_{j})\partial_{\theta}(dy^\alpha)\nonumber\\
&=& \partial_i\partial_j\phi^{\lambda}-\Gamma^{k}_{ij}\nabla_k\phi^{\lambda}+(\Gamma^{\lambda}_{\alpha\beta}\circ\phi)\nabla_i\phi^{\alpha}\nabla_j\phi^{\beta},
\end{eqnarray}
whence
\begin{equation*}
\nabla\nabla\phi=
(\nabla_{\partial_i}\nabla\phi)(\partial_j,dy^\lambda)dx^{i}\otimes dx^{j}\otimes \partial_{\lambda}|_{\phi}.
\end{equation*}
The tension field of $\phi$ with respect to the metrics $g$ and $\gamma$ is given as the trace of \eqref{coordinate_of_tensionfield}, and then
\begin{align}\label{coordinate_of_tensionfield-Aux}
\nonumber \tau_{g,\gamma}\phi &= \operatorname{tr}_{g}( \nabla\nabla\phi)\\
&= g^{ij}\Big(\partial_i\partial_j\phi^{\lambda}-\Gamma^{k}_{ij}\nabla_k\phi^{\lambda}+(\Gamma^{\lambda}_{\alpha\beta}\circ\phi)\nabla_i\phi^{\alpha}\nabla_j\phi^{\beta}\Big)\partial_{\lambda}|_{\phi}\\
\nonumber&=\Big( \Delta_{g}\phi^{\lambda}+(\Gamma^{\lambda}_{\alpha\beta}\circ\phi)g(\nabla\phi^{\alpha},\nabla\phi^{\beta})\Big)\partial_{\lambda}|_{\phi}. 
\end{align}
Notice that $\tau_{g,\gamma} \phi$ is a generalization of the Laplacian on $C^{\infty}(M)$. By definition, the map $\phi$ is harmonic if $\tau_{g,\gamma} \phi=0$. Even though definition~\eqref{coordinate_of_tensionfield-Aux} is well known, we need something a little more general, as we shall define now.

For $\omega:TM\longrightarrow \phi^{*}TN$, we write $\omega=\omega_i^{\lambda}dx^i\otimes\partial_\lambda$ to get
\begin{equation*}
\operatorname{div}_{g,\gamma}\omega \!:=\! g^{ij}\big(\nabla_{\partial_i}\omega\big)(\partial_j,dy^\lambda)\partial_{\lambda}\!|_{\phi}\!=\! g^{ij}\Big(\partial_i\omega_j^{\lambda}-\Gamma^{k}_{ij}\omega_k^{\lambda}+(\Gamma^{\lambda}_{\alpha\beta}\circ\phi)\nabla_i\phi^\alpha\omega_j^{\beta}\Big)\partial_{\lambda}|_{\phi}.
\end{equation*}

We shall use the inner product on the bundle $T^{*}M\otimes\phi^{*}TN$ induced by $g$ and $\gamma$ as follows
\begin{equation}\label{inner_product_on_induced_bundle}    \langle\nabla\phi,\nabla\phi\rangle_{_{T^{*}M\otimes\phi^{*}TN}}:=g^{ij}\phi^*\gamma_{\alpha\beta}\nabla_i\phi^{\alpha}\nabla_j\phi^{\beta}.
\end{equation}
Since there is no danger of confusion, we  shall write
\begin{equation*}
T(\nabla\phi, \nabla\phi):= T^{ij}\phi^*\gamma_{\alpha\beta}\nabla_i\phi^{\alpha}\nabla_j\phi^{\beta}
\end{equation*}
for any $2$-tensor $T$ on $M$, and the same notation $\langle\cdot, \cdot\rangle$ for the inner products on $M$, $N$ and $T^{*}M\otimes\phi^{*}TN$. Besides, for the sake of simplicity, we write
\begin{equation*}
\nabla\phi\otimes\nabla\phi(\partial_i,\partial_j):=\phi^*\gamma_{\alpha\beta}\nabla_i\phi^{\alpha}\nabla_j\phi^{\beta}.
\end{equation*}
As in \eqref{inner_product_on_induced_bundle}, we have
\begin{equation*}
\langle S, T \rangle = g^{ik} g^{jl} \phi^*\gamma_{\alpha \beta} S^\alpha_{ij} T^\beta_{kl}
\end{equation*}
for any \(S,T \in T^*M \otimes T^*M \otimes \phi^*TN\).

\section{\bf Evolution of the functional associated with the Ricci flow coupled with the harmonic map heat flow}\label{Sec:VWEGHYA}

In this section $g(t)$ stands for a one-parameter family of Riemannian metrics on an $m$-dimensional compact smooth manifold $M$ with boundary $\partial M$, and $\phi(t)$ a one-parameter family of smooth maps from $M$ to an $n$-dimensional Riemannian manifold $(N,\gamma)$, with $g(0)=g$ and $\phi(0)=\phi$. Moreover, consider the product 
$$\mathscr P(M,N):=\operatorname{met}(M)\times C^\infty (M)\times C^\infty(M,N),$$ 
where ${\rm met}(M)$ denotes the set of all Riemannian metrics on $M$.

We shall adopt the following notation. Given $(g,f,\phi)\in \mathscr P(M)$, take variations $(g_{ij}+th_{ij}, f+t\ell,\phi+t\vartheta),$ with $h_{ij} \in \Gamma(\operatorname{Sym}^2(T^{*}M)), \ell \in C^{\infty}(M)$ and $\vartheta \in C^{\infty}(M,N)$ with $\vartheta(x) \in T_{\phi(x)}N$. We denote by $\delta$ the derivative $\frac{d}{dt}|_{t=0}$, and then $\delta g=h$, $\delta f=\ell$, and $\delta \phi=\vartheta$. Moreover, we are using the weighted volume element $d\mu=e^{-f}dM,$ which is \emph{weighted measure-preserving} if and only if $\frac{\operatorname{tr}_{g}h}{2}-\ell=0$ on $M$, since $\delta(e^{-f} dM)=(\frac{\operatorname{tr}_{g}h}{2}-\ell)e^{-f}dM.$

For the sake of simplicity, we are writing $\gamma_{\alpha \beta}$ on $M$ instead of $\phi^{*} \gamma_{\alpha \beta}$.
With these notations in mind, we compute the variation of $\mathcal F^{\alpha}_\infty$ as follows.
\begin{proposition}\label{var:MWGHY}
Under weighted measure-preserving, we have
\begin{align*}
\delta \mathcal F^{\alpha}_\infty \!=&\! \int_{M}\!\! \Big( \langle -h,\operatorname{Ric}+\nabla^{2} f-\alpha \nabla \phi\otimes \nabla \phi\rangle +2\alpha\langle\tau_{g,\gamma}\phi-\left\langle\nabla f, \nabla \phi \right\rangle, \vartheta\rangle\Big)e^{-f}dM\\
&- \int_{\partial M} \big( h^{\hat{i}\hat{j}}\mathcal{A}_{\hat{i}\hat{j}} + h^{00}(H + e_0 f)\big) e^{-f} dA + 2\alpha \int_{\partial M}\langle\nabla_{0}\phi, \vartheta\rangle e^{-f} dA.
\end{align*}
\end{proposition}

\begin{proof}
First, note that by~\eqref{MGHYF} we can write
\begin{align*}
 \mathcal F^{\alpha}_\infty (g, f, \phi) = I_\infty(g, f) - \alpha  E(g, f, \phi),
\end{align*}
where $E(g, f, \phi) :=\int_{M} |\nabla \phi|^2 e^{-f}dM$. Thus, we can use Proposition~2 in \cite{John_Lott}, which guarantees that
\begin{align*}
\delta I_{\infty}=-\!\!\int_M \!\!h^{ij}\left(R_{ij}+\nabla_i \nabla_j f \right) e^{-f} dM \!-\!\int_{\partial M}\!\!\left(h^{\hat{i}\hat{j}} \mathcal{A}_{\hat{i} \hat{j}}+h^{00}\left(H+ e_0 f\right)\right) e^{-f} dA.
\end{align*}
Hence, it is enough to prove that
\begin{align*}
\delta E=&\int_{M}\Big(-h^{ij} \gamma_{\alpha\beta}\nabla_{i} \phi^{\alpha} \nabla_{j} \phi^{\beta} -2\langle\tau_{g,\gamma}\phi-\left\langle\nabla f, \nabla \phi \right\rangle, \vartheta\rangle \Big)e^{-f}dM\\
&- 2\int_{\partial M}\langle\nabla_0\phi, \vartheta\rangle e^{-f}dA.
\end{align*}
Indeed, notice that
\begin{align*}
\delta E(h, \ell, \vartheta)= \int_{M} \Big(\delta\big(|\nabla \phi|^2\big) + |\nabla \phi|^2\Big(\frac{\operatorname{tr}_{g}h}{2} - \ell\Big)\Big) e^{-f} dM
\end{align*}
and
\begin{align*}
\delta\big(|\nabla \phi|^2\big) = -h^{i j}\gamma_{\alpha\beta}\nabla_{i} \phi^{\alpha} \nabla_{j} \phi^{\beta}  +2 g^{i j}\gamma_{\alpha\beta} \nabla_{i} \vartheta^{\alpha} \nabla_{j} \phi^{\beta},
\end{align*}
where $\vartheta^{\alpha}=\vartheta\circ y_\alpha$. So, under weighted measure-preserving, we have
\begin{align*}
\delta E= \int_{M}\Big(-h^{i j}\gamma_{\alpha\beta}\nabla_{i} \phi^{\alpha} \nabla_{j} \phi^{\beta}e^{-f} +2 g^{i j}\gamma_{\alpha\beta} \nabla_{i} \vartheta^{\alpha} \nabla_{j} \phi^{\beta} e^{-f}\Big) dM,
\end{align*}
which is equivalent to
\begin{align}\label{EqAuxDiv}
\nonumber\delta E =&\int_{M}\Big(-h^{i j}\gamma_{\alpha\beta}\nabla_{i} \phi^{\alpha} \nabla_{j} \phi^{\beta}e^{-f} +2g^{ij}\partial_i \Big(\gamma_{\alpha\beta}\vartheta^{\alpha}e^{-f}\nabla_{j}\phi^{\beta}\Big)\\ 
&-2g^{ij}\partial_{i}\gamma_{\alpha\beta}\vartheta^{\alpha}\nabla_{j}\phi^{\beta}e^{-f}-2g^{ij}\gamma_{\alpha\beta}\vartheta^{\alpha}\partial_i\partial_j\phi^\beta e^{-f}\\
\nonumber&+2g^{ij}\gamma_{\alpha\beta}\vartheta^{\alpha}\nabla_{i} f\nabla_{j}\phi^{\beta}e^{-f}\Big)dM.
\end{align}
Now, note that
\begin{align*}
&g^{ij}\nabla_i \Big(\vartheta^{\alpha}e^{-f} \gamma\otimes\nabla \phi\Big)\big(\partial_\alpha, \partial_\beta, \partial_j, dy^{\beta}\big) \\
=& g^{ij}\partial_i \Big(\vartheta^{\alpha}\!e^{-f}\!\gamma\otimes\nabla\phi\big( \partial_\alpha, \partial_\beta,\partial_j, dy^{\beta}\big)\Big) \!\!-\!\!g^{ij}\vartheta^{\alpha}\! e^{-f}\!\gamma\otimes\nabla \phi\big(\nabla_{\phi_*\partial_i}\partial_\alpha, \partial_\beta,\partial_j, dy^{\beta}\big)\\
-&g^{ij}\vartheta^{\alpha}\ e^{-f}\!\gamma\otimes\nabla \phi\big(\partial_\alpha, \nabla_{\phi_*\partial_i}\partial_\beta,\partial_j, dy^{\beta}\big)\!\!-\!\!g^{ij}\vartheta^{\alpha} e^{-f}\gamma\otimes\nabla \phi\big(\partial_\alpha, \partial_\beta,\nabla_i\partial_j, dy^{\beta}\big)\\
-&g^{ij}\vartheta^\alpha e^{-f} \gamma\otimes\nabla\phi\big(\partial_\alpha, \partial_\beta,\partial_j, \nabla_{\phi_*\partial_i} dy^{\beta}\big).
\end{align*}
Using $\partial_{i}\gamma_{\alpha\beta} = \gamma(\nabla_{\phi_*\partial_i}\partial_\alpha, \partial_\beta)+\gamma(\partial_\alpha, \nabla_{\phi_*\partial_i}\partial_\beta)$ into the previous equation, we obtain from~\eqref{covariant_derivative2}, \eqref{Oldenburg1}, \eqref{EqAuxDiv} and Stokes' theorem,
\begin{align*}
\delta E =&\int_{M} \Big( - h^{ij}\gamma_{\alpha\beta}\nabla_{i} \phi^{\alpha} \nabla_{j} \phi^{\beta} +2g^{ij}\gamma_{\alpha\beta}\vartheta^{\alpha}\Gamma_{ij}^k\nabla_k\phi^{\beta}\\
&- 2g^{ij}\gamma_{\alpha\beta}\vartheta^{\alpha}(\Gamma_{\xi\theta}^\beta \circ \phi)\nabla_i\phi^{\xi}\nabla_j\phi^{\theta}\\ 
&-2g^{ij}\gamma_{\alpha\beta}\vartheta^{\alpha}\partial_i\partial_j\phi^\beta +2g^{ij}\gamma_{\alpha\beta}\vartheta^{\alpha}\nabla_{i} f\nabla_{j}\phi^{\beta}\Big)e^{-f}dM\\
&- 2\int_{\partial M} \gamma_{\alpha\beta}\vartheta^{\alpha} \nabla_0 \phi^{\beta} e^{-f} dA.
\end{align*}
So, by \eqref{coordinate_of_tensionfield-Aux}, it is immediate that
\begin{align*}
\delta E =&\int_{M} \Big( - h^{ij}\gamma_{\alpha\beta}\nabla_{i} \phi^{\alpha} \nabla_{j} \phi^{\beta} -2\langle\vartheta,\tau_{g,\gamma}\phi\rangle+ 2\big\langle\vartheta, \langle\nabla f, \nabla \phi\rangle\big\rangle \Big)e^{-f} dM\\
&- 2\int_{\partial M} \left\langle \vartheta,\nabla_0 \phi \right\rangle e^{-f} dA,
\end{align*}
which is enough to conclude the result of the proposition.
\end{proof}
\begin{remark} \label{rem:recover1}
By considering $M$ compact without boundary in Proposition~\ref{var:MWGHY}, we recover the results by Müller~\cite[Eq.~(3.1)]{muller2012ricci}, for $\phi\in C^\infty(M, N)$ with $M$ and $N$ being closed Riemannian manifolds and $N$ isometrically embedded into
Euclidean space $\mathbb R^d$; and by List~\cite{Bernhard_List}, for $\phi\in C^\infty(M)$. In the compact case with boundary, we also recover the results by Gomes and Hudson~\cite[Prop.~1]{Gomes_Hudson}, for $\phi\in C^\infty(M)$; and by Lott~\cite{John_Lott}, for $\phi$ constant.
\end{remark}
The next two corollaries provide the critical points of $\mathcal F^{\alpha}_\infty$ under weighted measure-preserving.
\begin{corollary}\label{Rem:WPM}
If the induced metric on $\partial M$ is fixed, then the critical points of $\mathcal F^{\alpha}_\infty$ under weighted measure-preserving are gradient steady solitons on $M$ that satisfy $H +  e_0 f = 0$ and $\nabla_{0}\phi = 0$ on $\partial M$. 
\end{corollary}

\begin{proof}
By hypotheses $\frac{\operatorname{tr}_{g}h}{2} - \ell = 0$ on $M$ and $h^{\hat i\hat j} = 0$ on $\partial M.$ Hence, by  Proposition~\ref{var:MWGHY} we obtain
\begin{align}\label{var:almost}
&\int_{M}\Big(\langle h,  \alpha \nabla\phi \otimes \nabla\phi - {\rm Ric} _g - \nabla^2_g f\rangle +2\alpha\langle\vartheta, \tau_{g,\gamma} \phi -  \langle \nabla \phi, \nabla f\rangle \rangle \Big) e^{-f} dM \nonumber \\
 &\quad+\int_{\partial M}\Big( 2\alpha\langle \vartheta, \nabla_{0} \phi\rangle - \langle h, (H + e_0 f) e_0^\flat \otimes e_0^\flat\rangle  \Big)e^{-f} dA = 0,
\end{align}  
for all 
$(h, \vartheta) \in\Gamma(\operatorname{Sym}^2(T^{*}M)) \times C^{\infty}(M,N),$
where $  `` ^\flat "$ stands for musical isomorphism. We assume 
$h$ and $\vartheta$ are compactly supported, so that
\begin{align*}
&\int_{M}\Big(\langle h,  \alpha \nabla\phi \otimes \nabla\phi - {\rm Ric} _g - \nabla^2_g f\rangle +2\alpha\langle\vartheta, \tau_{g,\gamma}\phi -  \langle \nabla \phi, \nabla f\rangle \rangle \Big) e^{-f} dM = 0.
\end{align*}  
Therefore $(g, f, \phi)$ must be a gradient steady soliton to the $(RH)_{\alpha}$ flow and then, again by~\eqref{var:almost} we get
\begin{align*}
\int_{\partial M}\Big( 2\alpha \langle \vartheta, \nabla_{0} \phi\rangle-\langle h, (H + e_0 f) e_0^\flat \otimes e_0^\flat\rangle\Big) e^{-f} dA = 0,
\end{align*} 
 for all 
$(h, \vartheta) \in\Gamma(\operatorname{Sym}^2(T^{*}M)) \times C^{\infty}(M,N).$ So, $H + e_0 f = 0$ and $\nabla_{0} \phi = 0$ on  $\partial M$.
\end{proof}
If we relax the fixed induced metric assumption on the boundary, then we obtain the next result. 
\begin{corollary}\label{Rem:WPM1}
If the induced metric on $\partial M$ is not fixed, then the critical points of $\mathcal F^{\alpha}_\infty$ under weighted measure-preserving are gradient steady solitons on $M$ with totally geodesic boundary satisfying the conditions $e_0 f = 0$ and $\nabla_{0}\phi = 0$ on $\partial M$. 
\end{corollary} 
\begin{proof}
As in the first part of the proof 
of Corollary~\ref{Rem:WPM}, we show that $(g,\phi, f)$ is a gradient steady soliton. Then
\begin{align}\label{alm:crit}
\int_{\partial M}\Big( 2\alpha \langle \vartheta, \nabla_{0} \phi\rangle - \langle h, \mathcal A - (H + e_0 f) e_0^\flat \otimes e_0^\flat\rangle \Big)e^{-f} dA = 0,
\end{align}  
for all 
$(h, \vartheta) \in\Gamma(\operatorname{Sym}^2(T^{*}M)) \times C^{\infty}(M,N)$. Since the induced metric on $\partial M$ is not fixed, we obtain $\mathcal A=0,$ $e_0 f = 0$ and $\nabla_{0}\phi = 0$ on $\partial M$.
\end{proof}

\begin{remark}
Corollaries~\ref{Rem:WPM} and \ref{Rem:WPM1} recover results by Gomes and Hudson~\cite{Gomes_Hudson}, for the case $\phi \in C^\infty(M);$ and by Lott~\cite[Cor.~4]{John_Lott}, for $\phi$ constant.
\end{remark}

\section{\bf The modified Ricci flow coupled with harmonic map heat flow}\label{MRFCHMHF}
To prove the main results of this article, we need to work in the following setting. We say that a family  $(g(t),\phi(t))$ evolves by the modified \emph{$(RH)_{\alpha}$ flow} if it satisfies the system 
\begin{subequations}
\begin{empheq}[left=\empheqlbrace]{align}
\frac{\partial}{\partial t}g &= - 2({\rm Ric}  + \nabla^2 f - \alpha \nabla\phi \otimes  \nabla\phi),\label{grad:form:eq1}\\
\frac{\partial}{\partial t}\phi &=  \tau_{g,\gamma}\phi - \langle\nabla \phi,\nabla f\rangle.\label{grad:form:eq2}
\end{empheq}
\end{subequations}
and
\begin{align}\label{weighted_pr_meas}
\frac{\partial}{\partial t}f =-R-\Delta f+\alpha |\nabla \phi|^2
\end{align}
in $M\times[0,T)$, with $H+ e_0 f=0$ and $\nabla_{0}\phi = 0$ on $\partial M.$ 

We can find motivations for considering the modified \((RH)_\alpha\) flow setting in Proposition~\ref{var:MWGHY} and its corollaries. This approach will be very useful in the study of mean curvature flow in the $(g(t),\phi(t))-(RH)_\alpha$ flow background, which is the main research object of this article. 

Notice that along the modified $(RH)_\alpha$ flow, the measure $e^{-f}dM$ remains fixed, since from~\eqref{grad:form:eq1} we have $h_{ij}= 2(- R_{ij} - \nabla_i \nabla_j f+\alpha\gamma_{\alpha\beta} \nabla_i \phi^{\alpha} \nabla_j \phi^{\beta})$, and then using~\eqref{weighted_pr_meas}, we obtain $\frac{\operatorname{tr}_{g}h}{2} - \ell = 0$ on $M.$

In what follows, we establish the tools for working on the modified $(RH)_\alpha$ flow setting. The first is the time-derivative of $\mathcal F^{\alpha}_\infty$ under this flow.

\begin{proposition}\label{alm:mon}
If $(g(t),\phi(t))_{t \in[0, T)}$ evolves by the modified $(RH)_{\alpha}$ flow, then 
\begin{align*}
\dfrac{d}{d t}\mathcal F^{\alpha}_\infty =& 2 \int_{M} \Big(|{\rm Ric}  + \nabla^2 f -\alpha \nabla\phi \otimes \nabla\phi|^2 +\alpha|\tau_{g,\gamma}\phi - \langle \nabla \phi, \nabla f \rangle|^2\Big) e^{-f}dM \\
& + 2 \int_{\partial M} \Big( \widehat{\Delta} H - 2 \langle \widehat{\nabla} f,  \widehat{\nabla} H \rangle + \mathcal{A}( \widehat{\nabla} f,  \widehat{\nabla} f) + \mathcal{A}^{\hat{i}\hat{j}} \mathcal{A}_{\hat{i}\hat{j}} H + \mathcal{A}^{\hat{i}\hat{j}}R_{\hat{i}\hat{j}} \\
& + 2 R^{0 \hat{i}}  \widehat{\nabla}_{\hat{i}} f -  \widehat{\nabla}_{\hat{i}} R^{0\hat{i}} - \alpha \mathcal{A}(\widehat{\nabla} \phi, \widehat{\nabla} \phi )\Big)e^{-f} d A.
\end{align*}
In particular, if both $\big(\!R_{\hat{i}\hat{j}} \!+\! \nabla_i\nabla_j f\!- \alpha\gamma_{\alpha\beta} \nabla_{\hat{i}} \phi^{\alpha} \nabla_{\hat{j}} \phi^{\beta}\big)|_{\partial M}$ and $\big(R_{\hat{i}0} \!+\! \nabla_{\hat{i}} \nabla_0 f\big)|_{\partial M}$ vanish, then the boundary integrand vanishes.
\end{proposition}

\begin{proof}
By~\eqref{grad:form:eq1} and~\eqref{grad:form:eq2}, we have
\begin{equation*}
h_{ij} = 2(\alpha\gamma_{\alpha\beta} \nabla_i \phi^{\alpha} \nabla_j \phi^{\beta}  - R_{ij} - \nabla_i\nabla_j f) \quad \hbox{and} \quad \vartheta = \tau_{g,\gamma} \phi- \langle \nabla \phi, \nabla f\rangle.
\end{equation*}
Proposition~\ref{var:MWGHY} implies
\begin{align*}
\dfrac{d}{d t}\mathcal F^{\alpha}_\infty =& 2 \int_{M} \Big(|{\rm Ric}  + \nabla^2 f -\alpha \nabla\phi \otimes \nabla\phi|^2 +\alpha|\tau_{g,\gamma}\phi - \langle \nabla \phi, \nabla f \rangle|^2\Big) e^{-f}dM \\
& + 2\int_{\partial M} \big( \mathcal{A}^{\hat{i}\hat{j}}( R_{\hat{i}\hat{j}} + \nabla_{\hat{i}}\nabla_{\hat{j}} f-\alpha\gamma_{\alpha\beta}\nabla_{\hat{i}} \phi^{\alpha} \nabla_{\hat{j}} \phi^{\beta}) \big) e^{-f} d A,
\end{align*}
where we have used that $H+e_0 f = 0$ and $\nabla_{0}\phi= 0$ on $\partial M$. On the other hand, Lemma~1 in Lott~\cite{John_Lott} guarantees that
\begin{align*}
&\mathcal{A}^{\hat{i}\hat{j}}\left(R_{\hat{i}\hat{j}}+\nabla_{\hat{i}}\nabla_{\hat{j}} f\right) e^{-f} -\widehat{\nabla}_{\hat{i}}\Big(\left(R^{\hat{i} 0}+\nabla^{\hat{i}} \nabla^0 f\right) e^{-f}\Big) \\
=&\Big(\widehat{\Delta} H-2\langle\widehat{\nabla} f, \widehat{\nabla} H\rangle+\mathcal{A}(\widehat{\nabla} f, \widehat{\nabla} f)+\mathcal{A}^{\hat{i}\hat{j}} \mathcal{A}_{\hat{i}\hat{j}} H+\mathcal{A}^{\hat{i}\hat{j}} R_{\hat{i}\hat{j}}\\
&+2 R^{0 \hat{i}} \widehat{\nabla}_{\hat{i}} f-\widehat{\nabla}_{\hat{i}} R^{0 \hat{i}}\Big)e^{-f},
\end{align*}
where $\nabla^{\hat{i}}\nabla^{0} f = g^{\hat{i}\hat{k}}g^{0i} \nabla_{\hat{k}}\nabla_if$. Then
\begin{align}\label{vanishes}
&\mathcal{A}^{\hat{i}\hat{j}}\big(R_{\hat{i}\hat{j}} + \nabla_{\hat{i}}\nabla_{\hat{j}} f- \alpha\gamma_{\alpha\beta} \nabla_{\hat{i}} \phi^{\alpha} \nabla_{\hat{j}} \phi^{\beta}\big) e^{-f} - \widehat{\nabla}_{\hat{i}}\Big(\left(R^{\hat{i} 0}+\nabla^{\hat{i}} \nabla^0 f\right) e^{-f}\Big)\nonumber\\ 
=&\Big(\widehat{\Delta} H-2\langle\widehat{\nabla} f, \widehat{\nabla} H\rangle+\mathcal{A}(\widehat{\nabla} f, \widehat{\nabla} f)+\mathcal{A}^{\hat{i}\hat{j}} \mathcal{A}_{\hat{i}\hat{j}} H+\mathcal{A}^{\hat{i}\hat{j}} R_{\hat{i}\hat{j}}+2 R^{0 \hat{i}} \widehat{\nabla}_{\hat{i}} f\nonumber\\
&-\widehat{\nabla}_{\hat{i}} R^{0 \hat{i}} - \alpha \mathcal{A}(\widehat{\nabla} \phi,\widehat{\nabla} \phi )\Big)e^{-f},
\end{align}
and from Stokes' theorem
\begin{align*}
\int_{\partial M}\!\!\widehat{\nabla}_{\hat{i}}\Big(\big(R^{\hat{i} 0}+\nabla^{\hat{i}} \nabla^0 f\big) e^{-f}\Big)dA=\! \int_{\partial M}g^{\hat i\hat j}\widehat{\nabla}_{\hat{i}}\Big(g^{0k}\big(R_{\hat{j} k}+\nabla_{\hat{j}} \nabla_k f\big) e^{-f}\Big)dA = 0,
\end{align*}
which is enough to obtain the first part of the proposition. In particular, if both $\big(R_{\hat{i}\hat{j}} + \nabla_{\hat{i}}\nabla_{\hat{j}} f- \alpha\gamma_{\alpha\beta} \nabla_{\hat{i}} \phi^{\alpha} \nabla_{\hat{j}} \phi^{\beta}\big)|_{\partial M}$ and $\big(R_{\hat{i}0} +\nabla_{\hat{i}}\nabla_0 f\big)|_{\partial M}$ vanish, then from equation~\eqref{vanishes} the boundary integrand vanishes.
\end{proof}

In our next result, we establish the evolution equations of the geometric quantities of $\partial M$ under the modified $(RH)_{\alpha}$ flow. For its proof, we shall need the following identity. 
\begin{align}\label{Simons_identity}
\nonumber\widehat{\nabla}_{\hat{i}}\widehat{\nabla}_{\hat{j}} H
=& (\widehat{\Delta} \mathcal{A})_{{\hat{i}}{\hat{j}}}\! + \! \widehat{\nabla}_{\hat{i}}R_{{\hat{j}}0} \!+\!\widehat{\nabla}_{\hat{j}}R_{{\hat{i}}0} \!-\!\nabla_0R_{{\hat{i}}{\hat{j}}} + \mathcal{A}^{\hat{k}}\!_{\hat{i}} R_{0{\hat{k}}0{\hat{j}}} +\mathcal{A}^{\hat{k}}\!_{\hat{j}} R_{0{\hat{k}}0{\hat{i}}}- \mathcal{A}_{{\hat{i}}{\hat{j}}}R_{0 0} \nonumber\\
&\!+ 2 \mathcal{A}^{{\hat{k}}{\hat{l}}}R_{{\hat{k}}{\hat{i}}{\hat{l}}{\hat{j}}} - HR_{0{\hat{i}}0{\hat{j}}}- H\mathcal{A}^{\hat{k}}\!_{\hat{i}} \mathcal{A}_{{\hat{j}}{\hat{k}}} + \mathcal{A}^{{\hat{k}}{\hat{l}}}\mathcal{A}_{\hat{k}\hat{l}}\mathcal{A}_{{\hat{i}}{\hat{j}}} + \nabla_0 R_{0{\hat{i}}0{\hat{j}}}.
\end{align}
Identity~\eqref{Simons_identity} has already been observed by Lott~\cite{John_Lott}. Its proof can be obtained from Simons~\cite{James_Simons} or, alternatively, from Huisken~\cite{Huisken1986}. Indeed, in our notations, Lemma~2.1 in~\cite{Huisken1986} becomes
\begin{align*}
\widehat{\nabla}_{\hat{i}}\widehat\nabla_{\hat{j}} H&= (\widehat{\Delta}\mathcal{A})_{\hat{i}\hat{j}} -H\mathcal{A}_{\hat{i}\hat{k}}\mathcal{A}^{\hat{k}}\!_{\hat{j}}+\mathcal{A}^{\hat{k}\hat{l}}\mathcal{A}_{\hat{k}\hat{l}}\mathcal{A}_{\hat{i}\hat{j}} - HR_{0\hat{i}0\hat{j}}  +\mathcal{A}_{\hat{i}\hat{j}}R_{0 \hat{k} 0}^{\hat{k}} - \mathcal{A}^{\hat{k}}\!_{\hat{j}} R_{\hat{k}\hat{l}\hat{i}}^{\hat{l}}\\
&\quad - \mathcal{A}^{\hat{k}}\!_{\hat{i}} R_{\hat{k}\hat{l}\hat{j}}^{\hat{l}}  + 2 \mathcal{A}^{\hat{k}\hat{l}}R_{\hat{k}\hat{i}\hat{l}\hat{j}} +\nabla_{\hat{j}}R_{0\hat{k}\hat{i}}^{\hat{k}}- \nabla_0 R_{\hat{i}\hat{k}\hat{j}}^{\hat{k}}  + \nabla_{\hat{i}}R_{0\hat{k}\hat{j}}^{\hat{k}}.
\end{align*}
Hence, \eqref{Simons_identity} follows from the equality $\nabla_{\hat{i}} R_{\hat{j}0}=\widehat{\nabla}_{\hat{i}} R_{\hat{j}0}-\mathcal{A}_{\hat{i}\hat{j}}R_{00}+{\mathcal{A}}^{\hat{k}}\!_{\hat{i}}R_{\hat{j}\hat{k}}.$

\begin{proposition}\label{Gradient_formulate}
If $(g(t),\phi(t))_{t \in[0, T)}$ evolves by the modified $(RH)_{\alpha}$ flow,  then the following evolution equations hold on $\partial M$
\begin{align}
\frac{\partial}{\partial t}g_{\hat{i}\hat{j}} &= -(\mathcal{L}_{\widehat{\nabla }f}g)_{\hat{i}\hat{j}}-2(R_{\hat{i}\hat{j}}-\alpha\gamma_{\alpha\beta}\widehat{\nabla}_{\hat{i}} \phi^{\alpha} \widehat{\nabla}_{\hat{j}} \phi^{\beta})-2H\mathcal{A}_{\hat{i}\hat{j}}, \label{flow_with_mean_curv}\\
\frac{\partial}{\partial t}\phi &=  \tau_{\widehat g,\gamma} \phi + \nabla_0\nabla_0 \phi -\mathcal{L}_{\widehat{\nabla }f}\phi, \label{List_heat_equation}\\
\frac{\partial}{\partial t}\mathcal{A}_{\hat{i}\hat{j}} &= (\widehat{\Delta} \mathcal{A})_{\hat{i}\hat{j}} - (\mathcal{L}_{\widehat{\nabla }f}\mathcal{A})_{\hat{i}\hat{j}} - 
\mathcal{A}^{\hat{k}}_{\hat{i}} R^{\hat{l}}_{\hat{k}\hat{l}\hat{j}} - \mathcal{A}^{\hat{k}}_{\hat{j}} R^{\hat{l}}_{\hat{k}\hat{l}\hat{i}} +2 \mathcal{A}^{\hat{k}\hat{l}} R_{\hat{k}\hat{i}\hat{l}\hat{j}} -2 H \mathcal{A}_{\hat{i}\hat{k}}\mathcal{A}^{\hat{k}}_{\hat{j}} \nonumber\\
&\quad + \mathcal{A}^{\hat{k}\hat{l}}\mathcal{A}_{\hat{k}\hat{l}}\mathcal{A}_{\hat{i}\hat{j}} + \nabla_0 R_{0 \hat{i} 0 \hat{j}} \label{sec_fund_evolution}
\end{align}
and
\begin{equation}\label{List_mean_evolution}
\dfrac{\partial}{\partial t}H \!=\! \widehat{\Delta}H \!-\! \langle \widehat{\nabla} f, \widehat{\nabla} H\rangle \!+\! 2 \mathcal{A}^{\hat{i}\hat{j}}R_{\hat{i}\hat{j}} \!+\! \mathcal{A}^{\hat{i}\hat{j}}\mathcal{A}_{\hat{i}\hat{j}}H \!+\! \nabla_0 R_{0 0}\!-\! 2\alpha \mathcal{A}(\widehat{\nabla} \phi, \widehat{\nabla}\phi),
\end{equation}
where $\widehat{\nabla}$ and $\widehat{\Delta}$ denote the gradient and Laplacian of smooth functions computed on the induced metric $\widehat{g}$ on $\partial M$, respectively. Besides, $\mathcal{L}_{\widehat{\nabla }f}\phi:=\langle\nabla\phi^\lambda,\widehat{\nabla}f\rangle\partial_\lambda$ and $\nabla_0\nabla_0\phi:=\nabla_0\nabla_0\phi^\lambda\partial_\lambda.$
\end{proposition}
\begin{proof} 
We start by substituting $\nabla_{\hat{i}}\nabla_{\hat{j}} f = \widehat{\nabla}_{\hat{i}} \widehat{\nabla}_{\hat{j}} f + H \mathcal{A}_{{\hat{i}}{\hat{j}}}$ (as $H + e_0 f = 0$) 
into equation~\eqref{grad:form:eq1} to get
\begin{align*}
\frac{\partial}{\partial t}g_{\hat{i}\hat{j}} 
&= - 2\big(R_{\hat{i}\hat{j}} + \widehat{\nabla}_{\hat{i}}  \widehat{\nabla}_{\hat{j}} f + H \mathcal{A}_{\hat{i}\hat{j}} - \alpha\gamma_{\alpha\beta} \widehat{\nabla}_{\hat{i}} \phi^{\alpha} \widehat{\nabla}_{\hat{j}} \phi^{\beta}\big),
\end{align*}
which is~\eqref{flow_with_mean_curv}. Next, by~\eqref{coordinate_of_tensionfield-Aux} and~\eqref{grad:form:eq2}, we have
\begin{align*}
\frac{\partial}{\partial t}\phi =& \tau_{g,\gamma} \phi - \langle \nabla\phi,\nabla f \rangle\\
=& g^{ij} \left( \partial_i\partial_j\phi^{\lambda}-\Gamma_{ij}^{k}\nabla_k\phi^{\lambda}+(\Gamma_{\alpha\beta}^{\lambda}\circ\phi)\nabla_i\phi^{\alpha}\nabla_j\phi^{\beta} \right)\partial_{\lambda}|_\phi \\
&- \langle \nabla\phi^\lambda,\widehat{\nabla} f +e_0 fe_0 \rangle\partial_\lambda|_\phi.
\end{align*}
Now, note that $\nabla_{0} \phi = 0$ implies $\nabla_0 \phi^\alpha = 0$ for all $\alpha$, and then, we obtain
\begin{align*}
\frac{\partial}{\partial t}\phi=& \Big(\Delta\phi^{\lambda}+g^{\hat{i}\hat{j}}(\Gamma_{\alpha\beta}^{\lambda}\circ\phi)\widehat\nabla_{\hat{i}}\phi^{\alpha}\widehat\nabla_{\hat{j}}\phi^{\beta}- \langle \nabla\phi^\lambda,\widehat{\nabla} f \rangle\Big)\partial_{\lambda}|_\phi\\
=& \Big(\widehat{\Delta}\phi^{\lambda}\!+\!\nabla_0\nabla_0\phi^{\lambda} \!+\! g^{\hat{i}\hat{j}}(\Gamma_{\alpha\beta}^{\lambda}\circ\phi)\widehat\nabla_{\hat{i}}\phi^{\alpha}\widehat\nabla_{\hat{j}}\phi^{\beta}- \langle \nabla\phi^\lambda,\widehat{\nabla} f \rangle\Big)\partial_{\lambda}|_{\phi}\\
=&\tau_{\widehat g,\gamma} \phi + \nabla_0 \nabla_0\phi -\mathcal{L}_{\widehat{\nabla }f}\phi,
\end{align*}
and proves \eqref{List_heat_equation}. To show~\eqref{sec_fund_evolution} we first observe that \eqref{grad:form:eq1} implies
\begin{align}\label{hAux-prop}
\frac{1}{2}h_{k\ell} = -\big(R_{k\ell} + \nabla_k \nabla_\ell f-\alpha\gamma_{\alpha\beta} \nabla_k \phi^{\alpha} \nabla_\ell \phi^{\beta}\big).
\end{align}
Moreover, we know that
\begin{align*}
\delta \mathcal{A}_{\hat{i}\hat{j}}= \frac{1}{2}(\nabla_{\hat{i}} h_{\hat{j} 0} + \nabla_{\hat{j}} h_{\hat{i}0} - \nabla_0 h_{\hat{i}\hat{j}}) + \frac{1}{2}h_{00}\mathcal{A}_{\hat{i}\hat{j}}.
\end{align*} 
From $\nabla_0 \phi = 0$ and \eqref{hAux-prop} we get $\frac{1}{2}h_{00}=-(R_{00} + \nabla_0 \nabla_0 f)$. Thus, 
\begin{align}\label{Aux-dt-Aij}
\dfrac{\partial}{\partial t} \mathcal{A}_{\hat{i}\hat{j}}
\nonumber=& -\nabla_{\hat{i}}\big(R_{{\hat{j}}0} + \nabla_{\hat{j}} \nabla_0 f - \alpha\gamma_{\alpha\beta} \nabla_{\hat{j}} \phi^{\alpha} \nabla_0 \phi^{\beta}\big) \\
&- \nabla_{\hat{j}}\big(R_{{\hat{i}}0} + \nabla_{\hat{i}} \nabla_0 f- \alpha\gamma_{\alpha\beta} \nabla_{\hat{i}} \phi^{\alpha} \nabla_0 \phi^{\beta}\big)\\
\nonumber&+\nabla_0\big(R_{{\hat{i}}{\hat{j}}} + \nabla_{\hat{i}}\nabla_{\hat{j}}f- \alpha\gamma_{\alpha\beta} \nabla_{\hat{i}} \phi^{\alpha} \nabla_{\hat{j}} \phi^{\beta}\big) -\big(R_{00} + \nabla_0 \nabla_0 f \big) \mathcal{A}_{{\hat{i}}{\hat{j}}} .
\end{align}
Now we will compute some terms of the previous equation. We start by observing that 
\begin{align}\label{Aux-ij0t}
\nabla_{\hat{i}} \nabla_{\hat{j}} \nabla_0 f = \widehat{\nabla}_{\hat{i}} \nabla_{\hat{j}} \nabla_0 f - \mathcal{A}_{\hat{i}\hat{j}} \nabla_0 \nabla_0 f + \mathcal{A}^{\hat{k}}\!_{\hat{i}} \nabla_{\hat{j}}\nabla_{\hat{k}} f,
\end{align}
which is a straightforward computation. Since $H + e_0f = 0$, we get
\begin{equation*}
\nabla_{\hat{j}}\nabla_{\hat{k}} f = \widehat{\nabla}_{\hat{j}} \widehat{\nabla}_{\hat{k}} f + H \mathcal{A}_{{\hat{j}}{\hat{k}}} \;\;\hbox{and}\;\;  \nabla_{\hat{j}} \nabla_0 f = -\widehat{\nabla}_{\hat{j}} H + \mathcal{A}^{\hat{k}}\!_{\hat{j}} \widehat{\nabla}_{\hat{k}} f.
\end{equation*}
Replacing the previous identities into \eqref{Aux-ij0t}, one has
 \begin{align*}
\nabla_{\hat{i}} \nabla_{\hat{j}} \nabla_0 f \!&=\! -\widehat{\nabla}_{\hat{i}} \widehat{\nabla}_{\hat{j}} H + \widehat{\nabla}_{\hat{i}} \Big( \mathcal{A}^{\hat{k}}\!_{\hat{j}} \widehat{\nabla}_{\hat{k}} f \Big) \!-\! \mathcal{A}_{{\hat{i}}{\hat{j}}} \nabla_0 \nabla_0 f \!+\! \mathcal{A}^{\hat{k}}\!_{\hat{i}} \widehat{\nabla}_{\hat{j}}\widehat{\nabla}_{\hat{k}} f \!+\! H \mathcal{A}^{\hat{k}}\!_{\hat{i}} \mathcal{A}_{{\hat{j}}{\hat{k}}}\\
\!&=\! -\widehat{\nabla}_{\hat{i}} \widehat{\nabla}_{\hat{j}} H \!+\! \Big(\widehat{\nabla}_{\hat{i}} \mathcal{A}^{\hat{k}}\!_{\hat{j}}\Big) \widehat{\nabla}_{\hat{k}} f \!+\! \mathcal{A}^{\hat{k}}\!_{\hat{j}}\widehat{\nabla}_{\hat{i}}\widehat{\nabla}_{\hat{k}} f \!-\! \mathcal{A}_{{\hat{i}}{\hat{j}}} \nabla_0 \nabla_0 f + \mathcal{A}^{\hat{k}}\!_{\hat{i}} \widehat{\nabla}_{\hat{j}}\widehat{\nabla}_{\hat{k}} f\\
 &\quad+ H \mathcal{A}^{\hat{k}}\!_{\hat{i}} \mathcal{A}_{{\hat{j}}{\hat{k}}}.
\end{align*}
Next, we note that
\begin{align*}
\nabla_0 \nabla_{\hat{i}}\nabla_{\hat{j}}f- \nabla_{\hat{j}} \nabla_{\hat{i}} \nabla_0 f = \nabla_0 \nabla_{\hat{j}} \nabla_{\hat{i}} f - \nabla_{\hat{j}} \nabla_0 \nabla_{\hat{i}} f &= -R_{0{\hat{j}}{\hat{k}}{\hat{i}}} \widehat{\nabla}^{\hat{k}} f - R_{0{\hat{j}}0{\hat{i}}}\nabla_0 f.
\end{align*}
Now, by a straightforward computation, we have
\begin{align*}
&\nabla_{e_0} \big(\gamma_{\alpha\beta}\nabla_{\hat{i}} \phi^{\alpha} \nabla_{\hat{j}} \phi^{\beta})(p)=\nabla_{\partial_0} \big(\gamma_{\alpha\beta}\nabla_{i} \phi^{\alpha} \nabla_{j} \phi^{\beta})(p)\\
&= \gamma_{\alpha\beta}\nabla_{\hat{i}} \nabla_0 \phi^{\alpha} \nabla_{\hat{j}} \phi^{\beta} + \gamma_{\alpha\beta}\nabla_{\hat{i}} \phi^{\alpha} \nabla_{\hat{j}} \nabla_0 \phi^{\beta},
\end{align*}
for all $p\in\partial M$. In the same way, we also obtain
\begin{align*}
\nabla_{\hat{j}} \Big(\gamma_{\alpha\beta}\nabla_{\hat{i}} \phi^{\alpha} \nabla_0 \phi^{\beta}\Big) &= \gamma_{\alpha\beta}\nabla_{\hat{i}} \nabla_{\hat{j}} \phi^{\alpha} \nabla_0 \phi^{\beta} + \gamma_{\alpha\beta}\nabla_{\hat{i}} \phi^{\alpha} \nabla_0 \nabla_{\hat{j}} \phi^{\beta}\\
&=\gamma_{\alpha\beta}\nabla_{\hat{i}} \phi^{\alpha} \nabla_0 \nabla_{\hat{j}} \phi^{\beta}
\end{align*}
and 
\begin{align*}
\nabla_{\hat{i}} \Big(\gamma_{\alpha\beta}\nabla_{\hat{j}} \phi^{\alpha} \nabla_0 \phi^{\beta}\Big) &= \gamma_{\alpha\beta}\nabla_{\hat{j}} \nabla_{\hat{i}} \phi^{\alpha} \nabla_0 \phi^{\beta} + \gamma_{\alpha\beta}\nabla_{\hat{j}} \phi^{\alpha} \nabla_0 \nabla_{\hat{i}} \phi^{\beta}\\
&=\gamma_{\alpha\beta}\nabla_{\hat{j}} \phi^{\alpha} \nabla_0 \nabla_{\hat{i}} \phi^{\beta}.
\end{align*}

Using all this into \eqref{Aux-dt-Aij}, we get
\begin{align*}
\dfrac{\partial}{\partial t} \mathcal{A}_{{\hat{i}}{\hat{j}}}  \!=& \widehat{\nabla}_{\hat{i}} \widehat{\nabla}_{\hat{j}} H - \big(\widehat{\nabla}_{\hat{i}} \mathcal{A}_{{\hat{k}}{\hat{j}}}  - R_{0{\hat{j}}{\hat{i}}k} \big) \widehat{\nabla}^{\hat{k}} f- \mathcal{A}^k\!_{\hat{i}} \widehat{\nabla}_{\hat{j}}\widehat{\nabla}_{\hat{k}} f - \mathcal{A}^k\!_{\hat{j}}\widehat{\nabla}_{\hat{i}}\widehat{\nabla}_{\hat{k}} f + R_{0{\hat{i}}0{\hat{j}}}H \\
& -\nabla_{\hat{i}}R_{{\hat{j}}0}  -\nabla_{\hat{j}}R_{{\hat{i}}0} +\nabla_0R_{{\hat{i}}{\hat{j}}} - \mathcal{A}_{{\hat{i}}{\hat{j}}}R_{0 0}   - H \mathcal{A}^{\hat{k}}\!_{\hat{i}} \mathcal{A}_{{\hat{j}}{\hat{k}}}.
\end{align*}
By Codazzi-Mainardi equation $R_{0{\hat{j}}{\hat{i}}{\hat{k}}}=\widehat{\nabla}_{\hat{i}} \mathcal{A}_{{\hat{j}}{\hat{k}}} - \widehat{\nabla}_{\hat{k}} \mathcal{A}_{{\hat{i}}{\hat{j}}}$ one has
\begin{align*}
\dfrac{\partial}{\partial t} \mathcal{A}_{{\hat{i}}{\hat{j}}} \!=& \widehat{\nabla}_{\hat{i}} \widehat{\nabla}_{\hat{j}} H - \Big(\widehat{\nabla}_{\hat{k}} \mathcal{A}_{{\hat{i}}{\hat{j}}}\Big) \widehat{\nabla}^{\hat{k}} f- \mathcal{A}^{\hat{k}}\!_{\hat{i}} \widehat{\nabla}_{\hat{j}}\widehat{\nabla}_{\hat{k}} f - \mathcal{A}^{\hat{k}}\!_{\hat{j}}\widehat{\nabla}_{\hat{i}}\widehat{\nabla}_{\hat{k}} f  + R_{0{\hat{j}}0{\hat{i}}}H \\
& -\nabla_{\hat{i}}R_{{\hat{j}}0}  -\nabla_{\hat{j}}R_{{\hat{i}}0} +\nabla_0R_{{\hat{i}}{\hat{j}}} - \mathcal{A}_{{\hat{i}}{\hat{j}}}R_{0 0}  - H \mathcal{A}^{\hat{k}}\!_{\hat{i}} \mathcal{A}_{{\hat{j}}{\hat{k}}}\\
=&\widehat{\nabla}_{\hat{i}} \widehat{\nabla}_{\hat{j}} H - \Big(\mathcal{L}_{\widehat{\nabla}f} \mathcal{A}\Big)_{{\hat{i}}{\hat{j}}}  -\nabla_{\hat{i}}R_{{\hat{j}}0}  -\nabla_{\hat{j}}R_{{\hat{i}}0} +\nabla_0R_{{\hat{i}}{\hat{j}}} - \mathcal{A}_{{\hat{i}}{\hat{j}}}R_{0 0}  + R_{0{\hat{i}}0{\hat{j}}}H \\
& - H \mathcal{A}^{\hat{k}}\!_{\hat{i}} \mathcal{A}_{{\hat{j}}{\hat{k}}}.
\end{align*}
From Simons' identity~\eqref{Simons_identity} we get
\begin{align*}
&\dfrac{\partial}{\partial t} \mathcal{A}_{{\hat{i}}{\hat{j}}}\\
=&(\widehat{\Delta} \mathcal{A})_{{\hat{i}}{\hat{j}}} - \Big(\mathcal{L}_{\widehat{\nabla}f} \mathcal{A}\Big)_{{\hat{i}}{\hat{j}}}  -(\nabla_{\hat{i}}R_{{\hat{j}}0} -\widehat{\nabla}_{\hat{i}}R_{{\hat{j}}0})  -(\nabla_{\hat{j}}R_{{\hat{i}}0} - \widehat{\nabla}_{\hat{j}}R_{{\hat{i}}0})  -2 \mathcal{A}_{{\hat{i}}{\hat{j}}}R_{0 0} \\
&+ \mathcal{A}^{\hat{k}}\!_{\hat{i}} R_{0{\hat{k}}0{\hat{j}}} + \mathcal{A}^{\hat{k}}\!_{\hat{j}} R_{0{\hat{k}}0{\hat{i}}} + 2 \mathcal{A}^{{\hat{k}}l}R_{{\hat{k}}{\hat{i}}{\hat{l}}{\hat{j}}}  - 2 H \mathcal{A}^{\hat{k}}\!_{\hat{i}} \mathcal{A}_{{\hat{j}}{\hat{k}}} + \mathcal{A}^{{\hat{k}}{\hat{l}}}\mathcal{A}_{{\hat{k}}{\hat{l}}}\mathcal{A}_{{\hat{i}}{\hat{j}}} + \nabla_0 R_{0{\hat{i}}0{\hat{j}}}.
\end{align*}
As $\nabla_{\hat{i}} R_{{\hat{j}}0} = \widehat{\nabla}_{\hat{i}} R_{{\hat{j}}0}  - \mathcal{A}_{{\hat{i}}{\hat{j}}}R_{00} + \mathcal{A}^{\hat{k}}\!_{{\hat{i}}}R_{{\hat{j}}{\hat{k}}}$ we conclude that
\begin{align*}
\dfrac{\partial}{\partial t} \mathcal{A}_{{\hat{i}}{\hat{j}}} &=(\widehat{\Delta} \mathcal{A})_{{\hat{i}}{\hat{j}}} - \Big(\mathcal{L}_{\widehat{\nabla}f} \mathcal{A}\Big)_{{\hat{i}}{\hat{j}}}  -\mathcal{A}^{\hat{k}}\!_{\hat{i}} R^{\hat{l}}\!_{{\hat{k}}{\hat{l}}{\hat{j}}}  -\mathcal{A}^{\hat{k}}\!_{\hat{j}} R^{\hat{l}}\!_{{\hat{k}}{\hat{l}}{\hat{i}}}  + 2 \mathcal{A}^{{\hat{k}}{\hat{l}}}R_{{\hat{k}}{\hat{i}}{\hat{l}}{\hat{j}}} - 2 H \mathcal{A}^{\hat{k}}\!_{\hat{i}} \mathcal{A}_{{\hat{j}}{\hat{k}}} \\
&\quad + \mathcal{A}^{{\hat{k}}{\hat{l}}}\mathcal{A}_{{\hat{k}}{\hat{l}}}\mathcal{A}_{{\hat{i}}{\hat{j}}} + \nabla_0 R_{0{\hat{i}}0{\hat{j}}},
\end{align*}
which is \eqref{sec_fund_evolution}. For finishing our proof, we will show~\eqref{List_mean_evolution}. For it, note that
\begin{align*}
\delta H = -h_{{\hat{i}}{\hat{j}}}\mathcal{A}^{{\hat{i}}{\hat{j}}} + g^{{\hat{i}}{\hat{j}}}\delta \mathcal{A}_{{\hat{i}}{\hat{j}}}
\end{align*}
and
\begin{align*}
g^{{\hat{i}}{\hat{j}}}(\mathcal{L}_{\widehat{\nabla}f} \mathcal{A})_{{\hat{i}}{\hat{j}}} -2 \mathcal{A}^{{\hat{i}}{\hat{j}}}\widehat{\nabla}_{\hat{i}}\widehat{\nabla}_{\hat{j}} f = \widehat{\nabla}_{\widehat{\nabla} f} (g^{{\hat{i}}{\hat{j}}} \mathcal{A}_{{\hat{i}}{\hat{j}}})= \langle \widehat{\nabla} f, \widehat{\nabla} H\rangle.
\end{align*}
So,
\begin{align*}
\dfrac{\partial}{\partial t}H 	=&2(R_{{\hat{i}}{\hat{j}}} + \widehat{\nabla}_{\hat{i}}\widehat{\nabla}_{\hat{j}} f + H \mathcal{A}_{{\hat{i}}{\hat{j}}}) \mathcal{A}^{{\hat{i}}{\hat{j}}} + g^{{\hat{i}}{\hat{j}}} \Big((\widehat{\Delta} \mathcal{A})_{{\hat{i}}{\hat{j}}} - \big(\mathcal{L}_{\widehat{\nabla}f} \mathcal{A}\big)_{{\hat{i}}{\hat{j}}}  -\mathcal{A}^{\hat{k}}_{\hat{i}} R^{\hat{l}}_{{\hat{k}}{\hat{l}}{\hat{j}}}   \\
& -\mathcal{A}^{\hat{k}}_{\hat{j}} R^{\hat{l}}_{{\hat{k}}{\hat{l}}{\hat{i}}}  + 2 \mathcal{A}^{{\hat{k}}{\hat{l}}}R_{{\hat{k}}{\hat{i}}{\hat{l}}{\hat{j}}}  - 2 H \mathcal{A}^{\hat{k}}_{\hat{i}} \mathcal{A}_{{\hat{j}}{\hat{k}}} + \mathcal{A}^{{\hat{k}}{\hat{l}}}\mathcal{A}_{{\hat{k}}{\hat{l}}}\mathcal{A}_{{\hat{i}}{\hat{j}}} + \nabla_0 R_{0{\hat{i}}0{\hat{j}}}  \Big) \\
& - 2\alpha \mathcal{A}(\widehat{\nabla} \phi, \widehat{\nabla} \phi)\\
=&2\mathcal{A}^{{\hat{i}}{\hat{j}}}R_{{\hat{i}}{\hat{j}}}  + 2H \mathcal{A}^{{\hat{i}}{\hat{j}}}\mathcal{A}_{{\hat{i}}{\hat{j}}}  + \widehat{\Delta} H     -\Big(g^{{\hat{i}}{\hat{j}}}\big(\mathcal{L}_{\widehat{\nabla}f} \mathcal{A}\big)_{{\hat{i}}{\hat{j}}} -2 \mathcal{A}^{{\hat{i}}{\hat{j}}}\widehat{\nabla}_{\hat{i}}\widehat{\nabla}_{\hat{j}} f\Big) \\
& - 2  \mathcal{A}^{{\hat{k}}{\hat{j}}} \mathcal{A}_{{\hat{j}}{\hat{k}}} H + \mathcal{A}^{{\hat{k}}{\hat{l}}}\mathcal{A}_{{\hat{k}}{\hat{l}}}H + \nabla_0 R_{00} - 2\alpha \mathcal{A}(\widehat{\nabla} \phi, \widehat{\nabla} \phi )\\
=& \widehat{\Delta}H - \langle \widehat{\nabla} f, \widehat{\nabla} H\rangle + 2 \mathcal{A}^{{\hat{i}}{\hat{j}}}R_{{\hat{i}}{\hat{j}}} + \mathcal{A}^{{\hat{i}}{\hat{j}}}\mathcal{A}_{{\hat{i}}{\hat{j}}}H + \nabla_0 R_{0 0} - 2\alpha \mathcal{A}(\widehat{\nabla} \phi, \widehat{\nabla}\phi).
\end{align*}
This finishes the proof.
\end{proof}

As a consequence of Proposition~\ref{Gradient_formulate}, we have the following refinement of the formula obtained in Proposition~\ref{alm:mon}.  
\begin{corollary}\label{key_prop_monotonicity}
If $(g(t),\phi(t))_{t\in[0,T)}$ evolves by the modified $(RH)_\alpha$ flow,  then the following identity holds
\begin{align*}
\dfrac{d}{d t}\mathcal F^{\alpha}_\infty =&2 \int_{M} \Big(\left|{\rm Ric}  + \nabla^2 f -\alpha \nabla\phi \otimes \nabla\phi\right|^2 + \alpha|\tau_{g,\gamma}\phi - \langle \nabla \phi, \nabla f\rangle|^2 \Big) e^{-f}dV \\
&+ 2 \int_{\partial M} \Big(\dfrac{\partial H}{\partial t}  - \langle \widehat{\nabla} f, \widehat{\nabla} H\rangle + \mathcal{A}(\widehat{\nabla} f, \widehat{\nabla} f) + 2 R^{0 {\hat{i}}}\widehat{\nabla}_{\hat{i}} f- \dfrac{1}{2} \nabla_0 R \\
&- HR_{00} +  \alpha \mathcal{A} (\widehat{\nabla} \phi, \widehat{\nabla} \phi)\Big)e^{-f} dA.
\end{align*}
In particular, if both $\big(\!R_{{\hat{i}}{\hat{j}}} \!+\! \nabla_{\hat{i}}\nabla_{\hat{j}}f\!- \alpha \gamma_{\alpha\beta} \nabla_{\hat{i}} \phi^{\alpha} \nabla_{\hat{j}} \phi^{\beta}\big)|_{\partial M}$ and $\big(\!R_{{\hat{i}}0} \!+\! \nabla_{\hat{i}} \nabla_0 f\big)|_{\partial M}$ vanish, then the boundary integrand vanishes.
\end{corollary}

\begin{proof}
From equation~\eqref{List_mean_evolution} of Proposition~\ref{Gradient_formulate}, the boundary integrand term of Proposition~\ref{alm:mon} can be rewritten as
\begin{align*}
&\widehat{\Delta} H - 2 \langle \widehat{\nabla} f,  \widehat{\nabla} H \rangle + \mathcal{A}( \widehat{\nabla} f,  \widehat{\nabla} f) + \mathcal{A}^{{\hat{i}}{\hat{j}}} \mathcal{A}_{{\hat{i}}{\hat{j}}} H + \mathcal{A}^{{\hat{i}}{\hat{j}}}R_{{\hat{i}}{\hat{j}}} + 2 R^{0 {\hat{i}}}  \widehat{\nabla}_{\hat{i}} f -  \widehat{\nabla}_{\hat{i}} R^{0{\hat{i}}}\\
&\quad - \alpha\mathcal{A}(\widehat{\nabla} \phi, \widehat{\nabla} \phi )\\
& = \dfrac{\partial H}{\partial t} -  \langle \widehat{\nabla} f,  \widehat{\nabla} H \rangle + \mathcal{A}( \widehat{\nabla} f,  \widehat{\nabla} f)  - \mathcal{A}^{{\hat{i}}{\hat{j}}}R_{{\hat{i}}{\hat{j}}} + 2 R^{0 {\hat{i}}}  \widehat{\nabla}_{\hat{i}} f -  \widehat{\nabla}_{\hat{i}} R^{0{\hat{i}}}  - \nabla_0 R_{00} \\
&\quad  + \alpha \mathcal{A}(\widehat{\nabla} \phi, \widehat{\nabla} \phi).
\end{align*}
Contracted Bianchi Identity and the fact that $\nabla_{\hat{i}} R_{{\hat{j}}0} = \widehat{\nabla}_{\hat{i}} R_{{\hat{j}}0} - \mathcal{A}_{{\hat{i}}{\hat{j}}}R_{00} + \mathcal{A}^k\!_{{\hat{i}}}R_{{\hat{j}}k}$ imply
\begin{align*}
\dfrac{1}{2} \nabla_0 R = \nabla_{\hat{i}} R^{{\hat{i}}0} + \nabla_0 R_{00} = \widehat{\nabla}_{\hat{i}} R^{{\hat{i}}0}  - HR_{00} + \mathcal{A}^{{\hat{i}}{\hat{j}}}R_{{\hat{i}}{\hat{j}}} + \nabla_0 R_{00}.
\end{align*}
The main result of the corollary follows from these two latter equations. If, in addition, both $R_{{\hat{i}}{\hat{j}}} + \nabla_{\hat{i}}\nabla_{\hat{j}}f- \alpha \gamma_{\alpha\beta} \nabla_{\hat{i}} \phi^{\alpha} \nabla_{\hat{j}} \phi^{\beta}$ and $R_{{\hat{i}}0} + \nabla_{\hat{i}} \nabla_0 f$ vanish on $\partial M$, then by Proposition~\ref{alm:mon} the integrand of $\partial M$, namely
\begin{align*}
\dfrac{\partial H}{\partial t}  - \langle \widehat{\nabla} f, \widehat{\nabla} H\rangle + \mathcal{A}(\widehat{\nabla} f, \widehat{\nabla} f) + 2 R^{0 {\hat{i}}}\widehat{\nabla}_{\hat{i}} f- \dfrac{1}{2} \nabla_0 R - HR_{00} +  \alpha \mathcal{A} (\widehat{\nabla} \phi, \widehat{\nabla} \phi)
\end{align*}
vanishes.
\end{proof}

\section{\bf Hypersurfaces in the Ricci flow coupled with harmonic map heat flow background} \label{se:Ricci}

In this section, we prove Theorems~\ref{principal_theorem} and \ref{Huisken_monotonicity}. For this, we shall need the following.

\begin{proposition}\label{mean_curv_flow_in_a_mod}
Let $M$ be an $m$-dimensional smooth manifold. Suppose $\mathscr F:=\{\Sigma_t\,;\, t\in[0,T)\}$ is a mean curvature flow in the $(g(t),\phi(t))$-$(RH)_\alpha$ flow on $M$ which satisfies $\nabla_0 \phi=0$ on $\Sigma_0,$ where $e_0$ is the  unit normal vector field on $\Sigma_0.$ Then, the following evolution equations hold 
\begin{align}
\frac{\partial}{\partial t}g_{{\hat{i}}{\hat{j}}} &= - 2\big(R_{{\hat{i}}{\hat{j}}}  - \alpha \gamma_{\alpha\beta}\widehat{\nabla}_{\hat{i}} \phi^{\alpha} \widehat{\nabla}_{\hat{j}} \phi^{\beta}\big)-2H\mathcal{A}_{{\hat{i}}{\hat{j}}},\label{List:together:mean:curv}\\
\frac{\partial}{\partial t}\phi &=  \tau_{\widehat g,\gamma} \phi + \nabla_0\nabla_0 \phi,\label{List:secondEq}\\
\frac{\partial}{\partial t}\mathcal{A}_{{\hat{i}}{\hat{j}}} &=
(\widehat{\Delta} \mathcal{A})_{{\hat{i}}{\hat{j}}}-\mathcal{A}^{\hat{k}}\!_{\hat{i}} R^{\hat{l}}_{{\hat{k}}{\hat{l}}{\hat{j}}}-\mathcal{A}^{\hat{k}}\!_{\hat{j}} R^{\hat{l}}_{{\hat{k}}{\hat{l}}{\hat{i}}}+2\mathcal{A}^{{\hat{k}}{\hat{l}}} R_{{\hat{k}}{\hat{i}}{\hat{l}}{\hat{j}}}
-2H  \mathcal{A}_{{\hat{i}}{\hat{k}}}\mathcal{A}^{\hat{k}}\!_{\hat{j}}  \label{sec_fund_variational}\\
&\quad + \mathcal{A}^{{\hat{k}}{\hat{l}}}\mathcal{A}_{{\hat{k}}{\hat{l}}}\mathcal{A}_{{\hat{i}}{\hat{j}}} + \nabla_0 R_{0 {\hat{i}} 0 {\hat{j}}} \nonumber
\end{align}
and
\begin{align}\label{mean_curv_background}
\dfrac{\partial}{\partial t}H = \widehat{\Delta}H + 2 \mathcal{A}^{ij}R_{ij} + \mathcal{A}^{ij}\mathcal{A}_{ij}H
+ \nabla_0 R_{0 0} -2\alpha \mathcal{A}(\widehat{\nabla}\phi, \widehat{\nabla}\phi).
\end{align}
\end{proposition}
\begin{proof}
First, assume $\Sigma_t=\partial X_t$ with each $X_t$ compact. Given an interval $[a, b]\subset [0,T)$ and the MCF of $\Sigma$ in the $(g(t),\phi(t))_{t \in[a, b]}-(RH)_\alpha$ flow background with $\nabla_a \phi =0$ on $\Sigma = \partial X_a.$ We can find a positive solution $u(t)=
e^{-f(t)}$ for
\begin{align}\label{assump_principal_thm}
\left\{
\begin{array}{rllcl}
\square^{*}_{g(t)}u &=&0   &\mbox{in}&  \bigcup_{t\in [a,b]}(X_t \times \{t\}) \subset M\times [a, b],\\[1ex]
e_t u &=&H_{g(t)}u   &\mbox{on}&  \bigcup_{t\in [a,b]}(\partial X_t \times \{t\}),
\end{array}
\right.
\end{align}
by solving it backwards in time from $t = b,$ where $\square^{*}_{g(t)}$ is defined as in~\eqref{adjoint:square}. Indeed, choosing diffeomorphisms from $r_t: X_a\to X_t$, we reduce the problem
of solving~\eqref{assump_principal_thm} to a parabolic equation on a fixed domain. For it, take $\widetilde g(t) =r_t^*g(t),$ $\widetilde \phi(t) =r_t^*\phi(t),$  $\widetilde f(t) =r_t^*f(t)$ and  $\widetilde u(t) =r_t^*u(t)$, it is straightforward to compute that 
\begin{align}
\label{Naza'sassumptionn}
\left\{
\begin{array}{rllcc}
 \square^{*}_{\widetilde{g}(t)}\widetilde u +\Big\langle \nabla_{\widetilde g(t)} \widetilde u, \dfrac{\partial r_t}{\partial t} \Big\rangle &=&0  &\mbox{in}& X_a \times[a, b],\\[2ex]
\widetilde e_t\widetilde u &=& H_{\widetilde g(t)}\widetilde u   &\mbox{on}&  \partial X_a \times[a, b]
\end{array}
\right.
\end{align}
which is equivalent to \eqref{assump_principal_thm}. Now, by using $s = b - t,$ we have that \eqref{Naza'sassumptionn} is equivalent to the following parabolic equation 
\begin{align}\label{assump_principal_thm*-eq}
\left\{\begin{array}{rllrl}\dfrac{\partial}{\partial s} \widetilde u(s) \!\!\!\!\!&=&\!\!\!\!\!  \Delta_{\widetilde g} \widetilde u  - R_{\widetilde g}\widetilde u +\alpha |\nabla_{\widetilde g} \widetilde \phi|^2 \widetilde u + \Big\langle \nabla_{\widetilde g} \widetilde u, \dfrac{\partial r_t}{\partial s}\Big\rangle   \!\!\!\!&\mbox{in}&\!\!\!\!   X_a \times[a, b],\\[1ex]\widetilde e_s\widetilde  u \!\!\!\!\!&=&\!\!\!\!\!H_{\widetilde g}\widetilde 
 u  &\mbox{on}&\!\!\!\!  \partial X_a \times [a,b].
 \end{array}\right.
\end{align}
It guarantees the existence of a solution $u(t) = e^{-f(t)}$ for~\eqref{assump_principal_thm}.

Thus, we can take a one-parameter family of
diffeomorphisms $\{\psi_t\}_{t \in [a, b]}$ generated by $\{-\nabla_{g(t)} f(t)\}_{t \in [a, b]}$,
with $\psi_a = \operatorname{Id}$. Then $\psi_t(X_a) = X_t$ for all $t.$
By setting $\widetilde{g}(t) =\psi_t^*g(t)$, $\widetilde{\phi}(t) =\psi_t^*\phi(t),$ $\widetilde{f}(t) =\psi_t^*f(t)$ and $\widetilde\gamma (t) = \psi_t^*\gamma(t),$
we have that $\widetilde{g}(t)$, $\widetilde{\phi}(t),$  $\widetilde{f}(t)$ and $\widetilde\gamma(t)$  are defined on $X_a$.
We claim that
\begin{align}\label{modified:RH:flow:hypersurface}
\left\{
\begin{array}{lll}
\dfrac{\partial }{\partial t} \widetilde{g}_{ij}&=&
-2 \big(\widetilde{R}_{ij} +  \widetilde{\nabla}_i \widetilde\nabla_j \widetilde{f} 
- \alpha\widetilde\gamma_{\alpha\beta} \widetilde{\nabla}_i \widetilde{\phi}^{\alpha}  \widetilde{\nabla}_j \widetilde{\phi}^{\beta}\big),\\[1ex]
\dfrac{\partial }{\partial t} \widetilde{\phi} &=& \tau_{\widetilde{g},\gamma} \widetilde{\phi}
- \langle \widetilde{\nabla} \widetilde{\phi} , \widetilde{\nabla}\widetilde{f} \rangle_{\widetilde{g}}
\end{array}
\right.
\end{align}
and
\begin{align}\label{backward:heat:equation:hypersurface}
\frac{\partial}{\partial t}\widetilde{f} =
- \Delta_{\widetilde{g}} \widetilde{f} - R_{\widetilde{g}}  + \alpha |\widetilde{\nabla} \widetilde{\phi}|_{\widetilde{g}}^2
\end{align}
in $ X_a \times [a,b]$ with $H_{\widetilde g} + e_af = 0$ and $\nabla_a \phi = 0$ on $ \partial X_a= \Sigma.$
Indeed,  to prove \eqref{modified:RH:flow:hypersurface}, we compute
\begin{align*}
\frac{\partial }{\partial t}\widetilde{g}_{ij} &= \psi^*_t \Big(\frac{\partial }{\partial t} g_{ij}\Big) + \psi^*_t\Big(\mathcal{L}_{\frac{d}{dt}\psi_t}g\Big)_{ij}\\
&=\psi^*_t \Big( -2(R_{ij}  - \alpha\gamma_{\alpha\beta} \nabla_i \phi^{\alpha} \nabla_j \phi^{\beta})\Big) - \psi^*_t\Big(\mathcal{L}_{\big(\nabla_{g(t)} f(t)\big)}g\Big)_{ij}\\
&= -2 \big(\widetilde{R}_{ij} +  \widetilde{\nabla}_i\widetilde\nabla_j\widetilde{f}  - \alpha\widetilde\gamma_{\alpha\beta} \widetilde{\nabla}_i \widetilde{\phi}^{\alpha}  \widetilde{\nabla}_j \widetilde{\phi}^{\beta}\big)
\end{align*}
and
\begin{align*}
\frac{\partial }{\partial t}\widetilde{\phi} &= \psi^*_t \Big(\frac{\partial }{\partial t} \phi\Big) + \psi^*_t\mathcal{L}_{\frac{d}{dt} \psi_t}\phi \\
&= \psi^*_t \Big(\tau_{g,\gamma} \phi\Big) - \psi^*_t\mathcal{L}_{\big(\nabla_{g(t)} f(t)\big)} \phi \\
&=\tau_{\widetilde{g},\gamma} \widetilde{\phi}  - \langle \widetilde{\nabla} \widetilde{\phi} , \widetilde{\nabla}\widetilde{f} \rangle_{\widetilde{g}}.
\end{align*}
To prove~\eqref{backward:heat:equation:hypersurface}, we use that $\Delta u = (|\nabla f|^2 - \Delta f) e^{-f}$ and~\eqref{assump_principal_thm} to obtain
\begin{align*}
\frac{\partial }{\partial t}\widetilde{f} &= \psi^*_t \Big(\frac{\partial }{\partial t} f\Big) + \psi^*_t\mathcal{L}_{\frac{d}{dt} \psi_t}f \\
&= \psi^*_t \Big( |\nabla f|^2 - \Delta f - R  + \alpha |\nabla \phi|^2\Big) - \psi^*_t\mathcal{L}_{\big(\nabla_{g(t)} f(t)\big)} f \\
&=  - \Delta_{\widetilde{g}} \widetilde{f} - R_{\widetilde{g}}  + \alpha |\widetilde{\nabla} \widetilde{\phi}|_{\widetilde{g}}^2 .
\end{align*}
For the boundary conditions, it is enough to note that $e_tu=H_{g(t)}u$ implies $e_tf(t)+ H_{g(t)} = 0,$  and then $0 =\psi^*_te_tf(t)+\psi^*_t H_{g(t)} =e_a \widetilde f(t) +H_{\widetilde g(t)}.$  Thus, $(\widetilde{g}(t), \widetilde{\phi}(t))$
evolves by the modified $(RH)_{\alpha}$ flow in $X_a \times [a,b]$, thus, we can apply Proposition~\ref{Gradient_formulate}
for the compact smooth manifold $X_a$ with boundary $\partial X_a,$ from which we obtain 
\begin{align*}
\frac{\partial }{\partial t} g_{{\hat{i}}{\hat{j}}}  &=  \frac{\partial }{\partial t} \Big((\psi^*_t)^{-1} \psi^*_t g_{{\hat{i}}{\hat{j}}}\Big) = \frac{\partial }{\partial t}\Big((\psi^*_t)^{-1} \widetilde{g}_{{\hat{i}}{\hat{j}}}\Big)
=(\psi^*_t)^{-1} \Big(\frac{\partial }{\partial t} \widetilde{g}_{{\hat{i}}{\hat{j}}}  + \Big(\mathcal{L}_{\frac{d}{dt}\psi_t^{-1}}\widetilde{g}\Big)_{{\hat{i}}{\hat{j}}} \Big)\\
&= - 2(R_{{\hat{i}}{\hat{j}}}  - \alpha\gamma_{\alpha\beta} \widehat{\nabla}_{\hat{i}} \phi^{\alpha} \widehat{\nabla}_{\hat{j}} \phi^{\beta}) - 2 H \mathcal{A}_{{\hat{i}}{\hat{j}}},
\end{align*}
on $\Sigma_t$ that is~\eqref{List:together:mean:curv}. Likewise, by equation~\eqref{List_heat_equation} one has
\begin{align*}
\frac{\partial }{\partial t} \phi &= (\psi^*_t)^{-1} \Big( \frac{\partial }{\partial t} \widetilde{\phi} +\mathcal{L}_{\frac{d}{dt}\psi_t^{-1}}\widetilde{\phi} \Big)
=\widehat{\tau}_{g,\gamma} \phi + \nabla_0 \nabla_0 \phi,
\end{align*}
which is~\eqref{List:secondEq}. Next, equation~\eqref{sec_fund_evolution} implies
\begin{align*}
\frac{\partial }{\partial t}\mathcal{A}_{{\hat{i}}{\hat{j}}} &=
(\psi^*_t)^{-1}\Big(\frac{\partial}{\partial t}\widetilde{\mathcal{A}}_{{\hat{i}}{\hat{j}}}+ \Big(\mathcal{L}_{\frac{d}{dt}\psi_t^{-1}}\widetilde{\mathcal{A}}\Big)_{{\hat{i}}{\hat{j}}}\Big) \\
&= (\widehat{\Delta} \mathcal{A})_{{\hat{i}}{\hat{j}}}  - \mathcal{A}^{\hat{k}}\!_{\hat{i}} R^{\hat{l}}_{{\hat{k}}{\hat{l}}{\hat{j}}} - \mathcal{A}^{\hat{k}}\!_{\hat{j}} R^{\hat{l}}_{{\hat{k}}{\hat{l}}{\hat{i}}}
+2 \mathcal{A}^{{\hat{k}}{\hat{l}}} R_{{\hat{k}}{\hat{i}}{\hat{l}}{\hat{j}}} -2H  \mathcal{A}_{{\hat{i}}{\hat{k}}}\mathcal{A}^{\hat{k}}\!_{\hat{j}}+\mathcal{A}^{{\hat{k}}l}\mathcal{A}_{{\hat{k}}{\hat{l}}}\mathcal{A}_{{\hat{i}}{\hat{j}}}\\
&\quad+ \nabla_0 R_{0 {\hat{i}} 0 {\hat{j}}}
\end{align*}
and from~\eqref{List_mean_evolution} we get
\begin{align*}
\frac{\partial }{\partial t} H &= (\psi^*_t)^{-1} \Big(\frac{\partial }{\partial t}H_{\widetilde{g}} + \mathcal{L}_{\frac{d}{dt} \psi_t^{-1}}H_{\widetilde{g}} \Big) \\
&= \widehat{\Delta}H + 2 \mathcal{A}^{{\hat{i}}{\hat{j}}}R_{{\hat{i}}{\hat{j}}} + \mathcal{A}^{{\hat{i}}{\hat{j}}}\mathcal{A}_{{\hat{i}}{\hat{j}}}H + \nabla_0 R_{0 0} -2\alpha \mathcal{A}(\widehat{\nabla}\phi, \widehat{\nabla}\phi).
\end{align*}
For finishing,  we observe that the result could be derived from a local calculation on $\Sigma_t,$ hence, it is also valid without the assumption that $\Sigma_t$ bounds a compact domain.
\end{proof}

\begin{remark}\label{without}
We  point out that~\eqref{List:together:mean:curv} hold  regardless the assumption $\nabla_0 \phi=0$ on $\Sigma_0$.
\end{remark}

\begin{remark} \label{rem:recover2}
If $M$ is the Euclidean space with its standard metric $g_0$, $g(t) = g_0$ and $\phi(t) = \phi$ is a constant,
then Eqs.~\eqref{List:together:mean:curv}, \eqref{sec_fund_variational} and \eqref{mean_curv_background}
are the same as in~\cite[Lem.~3.2, Thm.~3.4 and Cor.~3.5]{Huisken}, see also Mantegazza~\cite[Sect.~2.3]{Mantegazza}. Moreover, we recover Prop.~4 in Gomes and Hudson~\cite{Gomes_Hudson}, for $\phi \in C^\infty (M);$ and Prop.~4 in Lott~\cite{John_Lott}, for $\phi$ constant.
\end{remark}

\begin{proof}[\bf Proof of Theorem~\ref{principal_theorem}]
The hypotheses on $\{{\partial M}_t\,;\,t\in[0,T)\}$ and on $u$ allow us to use $\widetilde{g}(t)$, $\widetilde{\phi}(t)$ and $\widetilde{f}(t)$ on $M$ as in the proof of Proposition~\ref{mean_curv_flow_in_a_mod}.  In this way, the result follows immediately from Corollary~\ref{key_prop_monotonicity} and the fact that the identity 
$$\frac{\partial}{\partial t} H_{\widetilde{g}}=
\dfrac{\partial}{\partial t}H_g-\langle \widehat{\nabla} f,\widehat{\nabla} H\rangle$$ 
holds on $\partial M_t$ for all $t\in[0,T).$ 
\end{proof}

\begin{remark} \label{rem:recover3}
Theorem~\ref{principal_theorem} extends Theorem~1 in \cite{Gomes_Hudson}, which extends Theorem~1 in \cite{John_Lott}. Also, when $M$ is compact without boundary, it coincides with \cite[Eq.~(3.2)]{muller2012ricci}.
\end{remark}

We finalize this section by proving Theorem~\ref{Huisken_monotonicity}. First, we need to know how the area evolves under MCF in an $(RH)_\alpha$ flow background.

\begin{lemma}\label{Huisken:lem}
 Let $(\overline{g}(t),\overline{\phi}(t)),$ $\overline f$ and $\mathscr F:=\{\Sigma_t\}$ be as in the statement of Theorem~\ref{Huisken_monotonicity}. Then, the following equation holds on $\Sigma_t$ 
\begin{align*}
\dfrac{d}{dt}(d A_{\overline{g}}) = - \left(\overline{R}^i\!_i + H_{\overline{g}}^2 - \alpha|\widehat{\nabla}_{\overline{g}} \overline{\phi} |_{\overline{g}}^2\right) d A_{\overline{g}}.
\end{align*}
\end{lemma}
\begin{proof}
The lemma follows by using the well-known formula
\begin{align*}
\dfrac{d}{dt}(d A_{\overline{g}}) = \dfrac{1}{2}\mbox{tr}_{(\overline{g}_{ij}(t))} \Big( \dfrac{\partial}{\partial t} \overline{g}_{ij}\Big) dA_{\overline{g}}
\end{align*}
and equation~\eqref{List:together:mean:curv} in Proposition~\ref{mean_curv_flow_in_a_mod} (see also Remark~\ref{without}).
\end{proof}

\begin{proof}[\bf Proof of Theorem~\ref{Huisken_monotonicity}]
Lemma~\ref{Huisken:lem} and a straightforward computation  yield
\begin{align*}
\dfrac{d}{dt}\int_{\Sigma_t} e^{-\overline{f}} d A_{\overline{g}} &= - \int_{\Sigma_t} \Big(\frac{d}{dt}\overline{f} + \overline{R}^i\!_i + H_{\overline{g}}^2 - \alpha |\widehat{\nabla}_{\overline{g}} \overline{\phi} |_{\overline{g}}^2\Big) e^{-\overline{f} } d A_{\overline{g}}.
\end{align*}
By chain rule $\frac{d}{dt}\overline{f} = \frac{\partial}{\partial t}\overline{f} \frac{ dt}{ dt}  +  \overline g(t)(\nabla_{\overline g(t)}\overline{f}, \frac{\partial x}{\partial t})$ that implies
\begin{align*}
\dfrac{d}{dt}\int_{\Sigma_t} e^{-\overline{f}} d A_{\overline{g}}  &= - \int_{\Sigma_t} \Big(\dfrac{\partial}{\partial t}\overline{f} + H_{\overline{g}} e_t \overline{f}  + \overline{R}^i\!_i + H_{\overline{g}}^2 - \alpha |\widehat{\nabla}_{\overline{g}} \overline{\phi} |_{\overline{g}}^2\Big) e^{-\overline{f}} d A_{\overline{g}}.
\end{align*}

First, assume $(\overline{g}(t), \overline{\phi}(t))$ is a gradient steady soliton. In this case, we can take traces in the first equation 
of~\eqref{mod_grad_Ricci_soliton}  on $\Sigma_t$ to get
\begin{align*}
0=\overline{R}^i\!_i + \overline{\nabla}_i \overline{\nabla}^i \overline{f} -\alpha |\widehat{\nabla}_{\overline{g}} \overline{\phi}|_{\overline{g}}^2=\overline{R}^i\!_i + \widehat{\nabla}^i \widehat{\nabla}_i \overline{f} - H_{\overline{g}} e_t \overline{f} -\alpha |\widehat{\nabla}_{\overline{g}} \overline{\phi}|_{\overline{g}}^2.
\end{align*}
Then, using \eqref{self-solution}, we obtain
\begin{align*}
\dfrac{d}{dt}\int_{\Sigma_t} e^{-\overline{f}} d A_{\overline{g}}
&= - \int_{\Sigma_t} \Big(|\nabla_{\overline{g}} \overline{f}|_{\overline{g}}^2 - \widehat{\Delta}_{\overline{g}}\overline{f} +2 H_{\overline{g}} e_t \overline{f}  + H_{\overline{g}}^2 \Big)e^{-\overline{f}} d A_{\overline{g}} \\
&= - \int_{\Sigma_t} \Big( |\widehat{\nabla}_{\overline{g}} \overline{f}|_{\overline{g}}^2 +(e_t \overline{f})^2 - \widehat{\Delta}_{\overline{g}}\overline{f} + 2H_{\overline{g}} e_t \overline{f}  + H_{\overline{g}}^2\Big)e^{-\overline{f}} d A_{\overline{g}} \\
&= - \int_{\Sigma_t} \Big(H_{\overline{g}} +e_t \overline{f} \Big)^2 e^{-\overline{f}} d A_{\overline{g}},
\end{align*}
where in the second line we have used  the equality
$$\widehat{\Delta}_{\overline{g}} e^{- \overline{f}} = 
(|\widehat{\nabla}_{\overline{g}} \overline{f}|_{\overline{g}}^2 - 
\widehat{\Delta}_{\overline{g}}\overline{f})e^{-\overline{f}}$$ 
and Stokes' theorem. Since the boundary integrand on the right-hand side 
is nonnegative, we immediately have the result of the theorem for the steady case.

For the shrinking case, we claim that the function
$$(-\infty,T)\ni t\mapsto [4\pi(T - t)]^{-(m - 1 ) / 2}\int_{\Sigma_t} e^{-\overline{f}} dA_{\overline{g}}$$ 
is non-increasing during the flow. Indeed, as above, we take traces in the first equation 
of~\eqref{mod_grad_Ricci_soliton} on $\Sigma_t$ to obtain
\begin{align*}
 \dfrac{m - 1}{2(T - t)}=\overline{R}^i\!_i + 
\overline{\nabla}^i \overline{\nabla}_i \overline{f} -
\alpha |\widehat{\nabla}_{\overline{g}} \overline{\phi}|_{\overline{g}}^2 
= \overline{R}_i^i + \widehat{\nabla}^i \widehat{\nabla}_i \overline{f}- 
H_{\overline{g}} e_t \overline{f} -\alpha |\widehat{\nabla}_{\overline{g}} \overline{\phi}|_{\overline{g}}^2.
\end{align*}
Then,
\begin{align}\label{deriv_Huisken_shrk}
\nonumber &\dfrac{d}{dt}\Big([4\pi(T - t)]^{-(m - 1 ) / 2}\int_{\Sigma_t} e^{-\overline{f}} d A_{\overline{g}}\Big)\\
\nonumber &= - [4\pi(T - t)]^{-(m - 1 ) / 2}\int_{\Sigma_t} \Big( |\widehat{\nabla}_{\overline{g}} \overline{f}|_{\overline{g}}^2 +(e_t \overline{f})^2 - \nonumber \widehat{\Delta}_{\overline{g}}\overline{f} + 2H_{\overline{g}} e_t \overline{f}  + H_{\overline{g}}^2\\
&\quad+ \dfrac{m - 1}{2(T -t)}\Big)e^{-\overline{f}} d A_{\overline{g}} +\frac{m - 1 }{2} 4 \pi[4\pi(T - t)]^{-\frac{(m - 1 )}{2} - 1}\int_{\Sigma_t} e^{-\overline{f}} dA_{\overline{g}}\nonumber\\
&= -[4\pi(T - t)]^{-(m - 1 ) / 2} \int_{\Sigma_t} \Big(H_{\overline{g}} +e_t \overline{f} \Big)^2 e^{-\overline{f}} d A_{\overline{g}}.
\end{align}
This proves the claim and so the theorem for the shrinking case. Finally, in a similar way, one proves the expanding case.
\end{proof}

\begin{remark} \label{rem:recover5}
For the shrinking case in Theorem~\ref{Huisken_monotonicity},  we recover 
Huisken's monotonicity formula~\cite[Thm.~3.1]{Huisken1990}, 
by taking $M = \mathbb{R}^m$, $g_{ij}(\tau) = \delta_{ij}$, $\overline{f}(x, \tau) = |x|^2/4 \tau$ 
and $\overline{\phi}(\tau)= \phi$ to be a constant.
\end{remark}
\begin{remark} \label{rem:recover6}
We recover Huisken monotonicity-type formulas~\cite[Prop.~3.1]{magni2013flow} for hypersurface case of $M$ and~\cite[Prop.~8~and~Rmk.~5]{John_Lott}  by taking $\overline{\phi}(\tau)= \phi$ to be a constant. By taking $\overline \phi \in C^\infty(M)$ we recover~\cite[Thm.~2]{Gomes_Hudson}.
\end{remark}

\section{\bf Extension of Hamilton's differential Harnack expression}\label{Sec-EHDHE}
Here, we will see that the boundary integrand term of the time-derivative of $\mathcal F^{\alpha}_\infty$ provides an extension of Hamilton’s differential Harnack expression for mean curvature flow in Euclidean space to the more general context of mean curvature flow in the $(RH)_{\alpha}$ flow background.

Let $\mathscr F:=\{\Sigma_t\}$ be a family of mean curvature solitons in the $(\overline{g},\overline{\phi})-(RH)_{\alpha}$ flow background.
For the steady case, we have  
$\overline{R}_{{\hat{i}}{\hat{j}}} + \overline{\nabla}_{\hat{i}}\overline\nabla_{\hat{j}} \overline{f} - \alpha\gamma_{\alpha\beta} \overline{\nabla}_{\hat{i}} \overline{\phi}^{\alpha} \overline{\nabla}_{\hat{j}} \overline{\phi}^{\beta} 
= 0$ and  $\overline{R}_{{\hat{i}}0} + \overline{\nabla}_{\hat{i}}\overline\nabla_{0}\overline{f} - 
\alpha\gamma_{\alpha\beta} \overline{\nabla}_{\hat{i}} \overline{\phi}^{\alpha} \overline{\nabla}_0 \overline{\phi}^{\beta} =0$  on $\Sigma_t$. Then,
\begin{align}\label{soliton_restricted}
\overline{R}_{{\hat{i}}{\hat{j}}} + \widehat{\nabla}_{\hat{i}}\widehat\nabla_{\hat{j}} \overline{f} 
+ H_{\overline{g}} \mathcal{A}_{{\hat{i}}{\hat{j}}}- \alpha\gamma_{\alpha\beta}\widehat{\nabla}_{\hat{i}} \overline{\phi}^{\alpha}\widehat{\nabla}_{\hat{j}}\overline{\phi}^{\beta}=0,
\end{align}
and
\begin{align}\label{soliton_restricted2}
\overline{R}_{{\hat{i}}0} - \widehat{\nabla}_{\hat{i}} H_{\overline{g}} + \mathcal{A}^{\hat{k}}\!_{\hat{i}} \widehat{\nabla}_{\hat{k}} \overline{f} -\alpha\gamma_{\alpha\beta} \widehat{\nabla}_{\hat{i}} \overline{\phi}^{\alpha}\widehat{\nabla}_0 \overline{\phi}^{\beta}= 0 .
\end{align}

\begin{example}
For instance, consider $M = \mathbb{R}^m$, $\overline{g}(t)= \delta_{ij}$ and $\overline{\phi}(t)$ a constant, and let $L$ be a linear function on $\mathbb{R}^m$.  Defining $\overline{f} = L + t |\nabla L|^2$, we have that $\overline f$ satisfies~\eqref{self-solution}. Changing $\overline{f}$ to $-f$, equations~\eqref{soliton_restricted} and~\eqref{soliton_restricted2}  then become
\begin{align*}
\widehat{\nabla}_{\hat{i}}\widehat\nabla_{\hat{j}} \overline{f}-H \mathcal{A}_{{\hat{i}}{\hat{j}}}=0 \quad \hbox{and}\quad \widehat{\nabla}_{\hat{i}} H + \mathcal{A}^{\hat{k}}\!_{\hat{i}} \widehat{\nabla}_{\hat{k}} f = 0,
\end{align*}
respectively, which appear in~\cite[p.~219]{Hamilton} as equations for a translating soliton.
\end{example}

Consider a bounded domain $\Omega$  with smooth boundary $\partial\Omega:=\Sigma$ in Euclidean space $\mathbb{R}^m$, and   take a solution $u=e^{-f}$  to the  conjugate heat
equation~\eqref{back:heat:eq} in $\Omega\times[0,T)$
with $e_0 u=Hu$ on $\Sigma$. If $\mathscr{F}:=\{\Sigma_t\,;\, t\in[0,T)\}$ is a
mean curvature flow in a $(g(t),\phi(t))-(RH)_{\alpha}$ flow background
with $g(t)$ Ricci flat and $\nabla_0\phi = 0$ on $\Sigma$, then the boundary
integrand in Theorem~\ref{principal_theorem} becomes
\begin{align}\label{type:Harnack}
&\mathcal{Z}(V) + \alpha \mathcal{A}(\widehat{\nabla} \phi, \widehat{\nabla} \phi),
\end{align}
where $V=-\widehat{\nabla} f$ and $\mathcal{Z}(V):=\frac{\partial H}{\partial t}
+ 2\langle V, \widehat{\nabla} H\rangle + \mathcal{A}(V, V)$
is Hamilton’s differential Harnack expression for the case of mean
curvature flow in Euclidean space, which vanishes in the particular case of translating solitons (see~\cite[Def.~4.1 and Lem.~3.2]{Hamilton}).

The next result suggests an extension $\mathcal{Z}^{\alpha}_{\overline{g},\overline{\phi}}$ of $\mathcal{Z}$ for the more general case of MCF in the $(RH)_\alpha$ flow background, whose characterization of nullity should be on the steady case. For this, we observe that, if we consider a steady $(\overline{g}(t), \overline{\phi}(t))-(RH)_\alpha$ flow background on a smooth manifold $M$ with potential function $\overline{f}$, and $\Sigma$ is a mean curvature soliton at $t=0$, then its ensuing mean curvature flow $\{\Sigma_t \}$ consists of mean curvature solitons, and $\{\Sigma_t \}$ differs from $\{\psi_t(\Sigma)\}$ by hypersurface diffeomorphisms. In Section~\ref{sect:grad:sol}, we give a more general description that includes the shrinking and expanding soliton cases.

\begin{corollary}\label{type-Harnack-background}
Let $M$ be an $m$-dimensional smooth manifold and $(\overline{g}(t), \overline{\phi}(t))$ a gradient steady soliton on $M\times[0,T)$ with potential function $\overline{f}.$ Assume that  $\mathscr F:=\{\Sigma_t\,;\, t\in[0,T)\}$ is a mean curvature flow in the $(\overline g,\overline \phi)-(RH)_{\alpha}$ flow background which satisfies $H + e_0 f = 0$ and $\nabla_0\phi=0$ on $\Sigma_0,$ where $e_0$ is the  unit normal vector field on $\Sigma_0.$ Under these conditions, the  identity
\begin{align*}
\mathcal{Z}(-\widehat{\nabla}_{\overline{g}}\overline{f})+
 2 \overline{R}^{0 {\hat{i}}}\widehat{\nabla}_{\hat{i}} \overline{f}- \dfrac{1}{2} \overline{\nabla}_0 \overline{R}
 - H_{\overline{g}}\overline{R}_{00}+  
 \alpha \mathcal{A}(\widehat{\nabla}_{\overline{g}} \overline{\phi}, \widehat{\nabla}_{\overline{g}} \overline{\phi}) = 0
\end{align*}
holds for all $t\in[0,T)$,  where $\mathcal{A}$ and $\widehat{\nabla}_{\overline{g}}$ are as in Theorem \ref{principal_theorem}.
\end{corollary}

\begin{proof}
If $(\overline{g}(t), \overline{\phi}(t))$ is a gradient steady soliton on $M \times [0,T)$, then the positive function $u = e^{-\overline{f}(t)}$ on $\bigcup_{t\in [0,T)}(X_t \times \{t\}) \subset M\times [0, T)$ satisfies the conjugated heat equation~\eqref{assump_principal_thm} with $e_0 u = Hu$ and $\nabla_0\phi=0$ on $\partial X_0=\Sigma_0$, where the boundary conditions follows from the assumptions on $\Sigma_0$. To see this, first observe that $\Delta_{\overline{g}} u = (|\nabla_{\overline{g}} \overline{f}|_{\overline{g}}^2 - \Delta_{\overline{g}} \overline{f}) u$. Now taking traces in the first equation of~\eqref{mod_grad_Ricci_soliton} and using~\eqref{self-solution}, we obtain
 \begin{align*}
\frac{\partial}{\partial t} u = -u |\nabla_{\overline{g}} \overline{f}|_{\overline{g}}^2 = -\Delta_{\overline{g}} u+ R_{\overline{g}} u - \alpha|\nabla_{\overline{g}} \overline{\phi}|_{\overline{g}}^2u.
 \end{align*}
Thus, we can define $\widetilde{g}(t)$, $\widetilde{\phi}(t)$ and $\widetilde{f}(t)$ on $X_0$ as in the proof of Proposition~\ref{mean_curv_flow_in_a_mod}, so that $(\widetilde{g}(t), \widetilde{\phi}(t))$ evolves by the modified $(RH)_{\alpha}$ flow on $X_0 \times [0,T)$. Besides, again we use that $(\overline{g}(t), \overline{\phi}(t))$ is a gradient steady soliton and that $\nabla_0\phi=0$ on $\Sigma_0$ to get
\begin{align*}
 \big(\widetilde{R}_{{\hat{i}}{\hat{j}}} +\widetilde{\nabla}_{\hat{i}}\widetilde\nabla_{\hat{j}}\widetilde{f} - \alpha \gamma_{\alpha\beta} \widetilde{\nabla}_{\hat{i}} \widetilde{\phi}^{\alpha} \widetilde{\nabla}_{\hat{j}} \widetilde{\phi}^{\beta}\big)|_{\Sigma_0}\!=\!0 \,\,\hbox{and}\,\,  \big(\widetilde{R}_{{\hat{i}}0} + \widetilde{\nabla}_{\hat{i}}\widetilde\nabla_0\widetilde{f}\big)|_{\Sigma_0}\!=\!0.
\end{align*}
As in the proof of Theorem~\ref{principal_theorem}, the result of the corollary follows from Corollary~\ref{key_prop_monotonicity} 
and the identity 
$$\dfrac{\partial}{\partial t} H_{\widetilde{g}} = 
\dfrac{\partial}{\partial t} H_{\overline{g}} -
\langle \widehat{\nabla}_{\overline{g}} \overline{f}, \widehat{\nabla}_{\overline{g}} H_{\overline{g}}\rangle_{\overline{g}}.$$
This completes the proof.
\end{proof}

\begin{remark} \label{rem:recover4}
Suppose $M = \mathbb{R}^m$, $\overline{g}_{ij}(t) = \delta_{ij}$ and $\overline{\phi}(t) = \phi$ is a constant. 
Let $L$ be a linear function on $\mathbb{R}^m$ and define 
$\overline{f} = L + t |\nabla L|^2$. Letting $V(t) = 
-\widehat{\nabla} \overline{f}$, Corollary~\ref{type-Harnack-background}  coincides with~\cite[Lem.~3.2]{Hamilton}.
\end{remark}

\section{\bf Characterization of mean curvature solitons}\label{sect:grad:sol}
In this section, we show how to construct a family of mean curvature solitons and establish a characterization of such a family. For it, let $M$ be an $m$-dimensional smooth manifold, and let $(\overline{g}(t), \overline{\phi}(t))$ be a gradient soliton to the $(RH)_{\alpha}$ flow on $M$ for some initial value $(g,\phi)$ and with potential function $\overline f=\psi^*_tf$, where $\{\psi_t\}$ is the smooth one-parameter family of diffeomorphisms of $M$ generated by $Y_t = \frac{\nabla_g f}{\sigma(t)}$, with $\sigma(t) = \kappa(T-t)$ and $\psi_{T-\kappa} = \operatorname{Id},$ where $\kappa = 1$ in the shrinking case (for $t \in (-\infty, T$)), $\kappa = -1$ in the expanding case (for $t \in (T,+\infty)$) and $\sigma(t) = 1$ in the steady case (for $t \in  \mathbb{R})$ with  $\psi_{0} = \operatorname{Id}.$ 

Given an $(m-1)$-dimensional compact smooth manifold $\Sigma$ without boundary, let $\{x(\,\cdot\,,t)\}$ be a smooth one-parameter family of immersions of $\Sigma$ into $M$, where $x(\,\cdot\,,t):=\psi\big(\,\cdot\,,-t +2(T -\kappa)\big)$ and $x(\,\cdot\,,t):=\psi(\,\cdot\,,-t)$ in the steady case. Note that $x(\,\cdot\,,T-\kappa)=\psi(\,\cdot\,,T-\kappa)=\operatorname{Id}$ and $x(\,\cdot\,,0)=\psi(\,\cdot\,,0)=\operatorname{Id}$. Moreover, when considering $x(\cdot, t) :=\psi\big(\cdot,-t+2(T-\kappa)\big),$ we are assuming $t \in \big(2(T-1),T\big)$ in the shrinking case,  $t \in \big(T, 2(T+ 1)\big)$ in the expanding case, and $t\in \mathbb{R}$ in the steady case. For each $t,$ set  $x_t = x(\,\cdot\,, t),$ $\Sigma_t$ for the hypersurface $x_t(\Sigma)$ of $(M,\overline g(t)),$ i.e., $\Sigma_t:=(x_t(\Sigma), \overline g(t)),$ and $\mathscr G :=\{\Sigma_t\}.$ In particular, if $\mathscr G$ evolves by MCF in the $(\overline g, \overline \phi)-(RH)_{\alpha}$ flow background on $M$, then it is a family of mean curvature solitons. 
Indeed, since $\overline g(t)=\sigma(t)\psi_t^*g$, we have $\nabla_g  f= \sigma(t)\nabla_{\overline g(t)}  \overline f$, and then
\begin{align*}
 H(p,t) &= \overline g(t)\Big( \frac{\partial}{\partial t} x(p,t),  e(p, t)\Big) =-\overline g(t)\Big(  \frac{ \nabla_g  f(p)}{\sigma(t)},  e(p, t)\Big) \\
 &= -\overline g(t)\Big(   \nabla_{\overline g(t)}  \overline f(p),  e(p, t)\Big) =- e(p, t) \overline f(p).
\end{align*}
It proves our claim. 

Theorem~\ref{NazasCharacterization} states that if $\Sigma$ is an $f$-minimal hypersurface of $(M,g)$, then $\mathscr G$ is a family of mean curvature solitons in the $(\overline g, \overline \phi)-(RH)_{\alpha}$ flow background on $M$. Moreover, any family $\mathscr F$ of mean curvature solitons in the $(\overline g, \overline \phi)-(RH)_{\alpha}$ flow background on $M$ is given by $\mathscr G$ up to reparametrization, as proved below.

\begin{proof}[\bf Proof of Theorem~\ref{NazasCharacterization}]\label{test2}
Let $\Sigma$ be a hypersurface of $(M,g)$ satisfying $H + e_0 f = 0$ on $\Sigma,$ where $e_{0}$ is the  unit normal vector field on $\Sigma$. Take $\mathscr G =\{\Sigma_t\}$ the smooth one-parameter family of isometric immersions of $\Sigma$ into $M$ as above, so that $e_0 = \sqrt{\sigma (t)} e(\,\cdot \,, t),$ and then $ \mathcal A_{e_0} = \sqrt{\sigma (t)}  \mathcal A_{e(\cdot,t)} $ that implies $H = \sqrt{\sigma(t)} H(\cdot, t).$ So, $H(\,\cdot\,,t) + e(\,\cdot\,, t) \overline f=0$. Thus,
\begin{align*}
\Big(\frac{\partial}{\partial t}  x (\,\cdot\,, t)\Big)^\perp &= \overline g(t)\Big(\frac{\partial}{\partial t}  x (\,\cdot\,, t), e(\,\cdot\,,t)\Big)e(\,\cdot\,,t)
= -\overline g(t)\Big(  \frac{\nabla_g f}{\sigma(t)}, e(\,\cdot\,, t) \Big)e(\,\cdot\,,t) \\
&= -\overline g(t)\Big(   \nabla_{\overline g(t)}  \overline f,  e(\,\cdot\,, t)\Big) e(\,\cdot\,,t)=- e(\,\cdot\,, t) (\overline f) e(\,\cdot\,, t)  \\
&= H(\,\cdot\,,t) e(\,\cdot\,, t).   
\end{align*}
Now, we affirm that if a smooth family of hypersurfaces $\Sigma_t = x_t(\Sigma)$ satisfies $\langle\frac{\partial}{\partial t} x(p,t), e(p,t)\rangle \!=\!  H(p,t)$, then it can be everywhere locally reparametrized to a mean curvature flow. Indeed, if $\frac{\partial}{\partial t} x(p,t) = H(p,t) e(p,t) + X(p,t), $ where $X(p,t) \in d x_t(Tp\Sigma)\,\,\,\forall p \in \Sigma,$ take $\{\varphi_t\}$ the smooth one-parameter family of diffeomorphisms of $\Sigma$ generated by $Y(p,t) =  -[dx_t]^{-1}(X(p,t))$ and then consider the reparametrization $\widetilde x(p,t) = x(\varphi_t (p), t).$ By a straightforward computation  
 $\{\widetilde \Sigma _t := \widetilde x_t(\Sigma)\}$ evolves by MCF  in the $(\overline g, \overline \phi)-(RH)_{\alpha}$ flow background on $M$. Finally, by a simple analysis of this proof, we also show that any family $\mathscr F$ of mean curvature solitons is given by $\mathscr G$ up to reparametrization.
\end{proof}

\begin{remark}
The previous theorem recovers Thm.~3 in Gomes and Hudson~\cite{Gomes_Hudson}, for $\phi \in C^\infty(M)$; and Prop.~4.3 in Yamamoto~\cite{yamamoto2020meancurvature} in the case of gradient shrinking Ricci soliton and  $\phi$ constant.
\end{remark}

\section{\bf How to construct mean curvature solitons in the Ricci flow coupled with harmonic map heat flow background}\label{section:how:to:construct}

In this section, we show how to obtain a self-similar solution to the Ricci flow coupled with harmonic map heat flow, and how to obtain mean curvature solitons for MCF in a Ricci flow coupled with harmonic map heat flow background. For explicit examples of mean curvature solitons for MCF in a Ricci flow background, see Yamamoto~\cite{yamamoto2018examples}. For explicit examples of mean curvature solitons for MCF in an extended Ricci flow background, see Gomes and Hudson~\cite{Gomes_Hudson}.

Let $g_{ij} = \frac{1}{F^2}\delta_{ij}$ be a Riemannian metric on $\mathbb{R}^{m}$, and let $\gamma_{\alpha\beta} = \frac{1}{G^2}\delta_{\alpha\beta}$ be a Riemannian metric on $\mathbb{R}^{n}$, where  
$F$ and $G$ are nonzero smooth functions on $\mathbb{R}^{m}$ and $\mathbb{R}^n$, respectively. Consider 
\begin{subequations}\label{Bizu}
\begin{empheq}[left=\empheqlbrace]{align}
&\operatorname{Ric} + \nabla^2 f - \alpha \nabla \phi  \otimes \nabla \phi = \lambda g,\label{Bizu1}\\[1ex]
&\tau_{g,\gamma} \phi = \langle \nabla f, \nabla \phi \rangle\label{Bizu2}.
\end{empheq}
\end{subequations}
Since the metric $g_{ij}$ is conformal to $\delta_{ij}$, it is well known (see, e.g., \cite{besse2007einstein}) that 
\begin{align*}(\operatorname{Ric})_{ij} &=\frac{1}{F^2}\Big((m - 2)F  F_{x_ix_j} + \Big(F\sum_k F_{x_kx_k} - (m-1) \sum_k  F_{x_k}^2\Big)\delta_{ij}\Big)\\
(\nabla^2 h)_{ij} &= h_{x_ix_j} +\frac{F_{x_j}}{F} h_{x_i} +\frac{F_{x_i}}{F}
h_{x_j} \quad \forall i \neq j\\
(\nabla^2 h)_{ii} &= h_{x_ix_i} + 2\frac{F_{x_i}}{F} h_{x_i} -\sum_k\frac{F_{x_k}}{F} h_{x_k} \quad \forall i \\
\nabla h &= F^2 \sum_i h_{x_i}\partial_i\\
\langle\nabla\phi,\nabla f\rangle&= F^2\sum_k f_{x_k} \phi^{\lambda}_{x_k}\partial_{\lambda}|_{\phi}
\end{align*}
for any smooth functions $F$ and $h$ on $\mathbb{R}^{m}$. Hence,
$$\Delta h = F^2 \Big(\sum_k h_{x_kx_k} + (2 - m) \frac{1}{F} \sum_kF_{x_k}h_{x_k}\Big).$$
Moreover,
\[\nabla \phi  \otimes \nabla \phi(\partial_i,\partial_j) =\gamma_{\alpha\beta} \circ \phi \nabla_i\phi^{\alpha}\nabla_j\phi^{\beta} = \frac{1}{(G \circ \phi)^2} \phi^{\alpha}_{x_i}\phi^{\alpha}_{x_j}. \]
It is also known that
\[
\Gamma_{\alpha\beta}^\theta = -\delta_{\beta \theta} \frac{G_{y_\alpha}}{G} - \delta_{\theta\alpha} \frac{G_{y_\beta}}{G} + \delta_{\alpha\beta} \frac{G_{y_\theta}}{G}.
\] 
Thus,
\begin{align}\label{tensor:field:Euclidean}
&\tau_{g,\gamma} \phi\nonumber\\
=& \Big( \Delta\phi^\theta+ (\Gamma^{\theta}_{\alpha\beta}\circ\phi)g(\nabla\phi^{\alpha},\nabla\phi^{\beta})\Big)\partial_{\theta}|_{\phi}\nonumber\\
=& \Bigg( F^2 \Big(\sum_k \phi^{\theta}_{x_kx_k} + (2 - m) \frac{1}{F} \sum_kF_{x_k}\phi^{\theta}_{x_k}\Big)+ F^2(\Gamma^{\theta}_{\alpha\beta}\circ\phi) \sum_k\phi_{x_k}^{\alpha}\phi^{\beta}_{x_k}\Bigg)\partial_{\theta}|_{\phi}\nonumber\\
=& \Bigg( F^2 \Big(\sum_k \phi^{\theta}_{x_kx_k} + (2 - m) \frac{1}{F} \sum_kF_{x_k}\phi^{\theta}_{x_k}\Big)\nonumber\\
&+ F^2\Big(-\delta_{\beta \theta} \frac{G_{y_\alpha}}{G} - \delta_{\theta\alpha} \frac{G_{y_\beta}}{G} + \delta_{\alpha\beta} \frac{G_{y_\theta}}{G}\Big)\circ \phi \sum_k\phi_{x_k}^{\alpha}\phi^{\beta}_{x_k}\Bigg)\partial_{\theta}|_{\phi}.
\end{align}
 
Now we show how to find solutions of Eq.~\eqref{Bizu1} (as well of~\eqref{Bizu2}) of the form 
$f(r)$ and $\phi(r)$, where $r  = \|x\|$ on $\mathbb R^m\setminus \{0\}$.

\begin{proposition}\label{Airtonvaivai}
Consider $\mathbb R^m$ with the metric $g_{ij} = \frac{1}{F^2}\delta_{ij}$, and $\mathbb R^n$ with the metric $\gamma_{\alpha\beta} = \frac{1}{G^2}\delta_{\alpha\beta}$  for some nonzero smooth functions $F$ on $\mathbb R^m$ depending only on $r  = \|x\|$ and $G$ on $\mathbb R^n$ depending only on $\rho  = \|y\|.$ We can obtain smooth functions  $f (r)$ and maps $ \phi(r)$ satisfying~\eqref{Bizu1}  (as well~\eqref{Bizu2}),  by means of the system
\begin{align*}
\begin{cases}
\dfrac{(2m - 3)F'} {rF}+ \dfrac{f'}{r} + \dfrac{F''}{F}   - (m - 1)\Big(\dfrac{F'}{F}\Big)^2 -\dfrac{F'}{F}f'  = \dfrac{\lambda}{F^2}\\[1ex]
{\phi^\theta}'' \!+\! \Big(\!\frac{m - 1}{r}  \!-\! (m-2)\frac{F'}{F} - f'\!\Big){\phi^\theta}' \!+\!  {\phi^\alpha}' {\phi^\beta}'\Big(\!\delta_{\alpha\beta} \frac{\dot G{y_\theta}}{G\rho}-
\delta_{\beta \theta} \frac{\dot G{y_\alpha}}{G\rho} \!-\! \delta_{\theta\alpha} \frac{\dot G{y_\beta}}{G\rho}\!\Big)\circ \phi \!=\!0
\end{cases}    
\end{align*}
for all $x \neq 0$ and $y \neq 0,$ where the superscripts $'$ and $\overset{\cdot}{}$ denote the derivative with respect to $r$ and $\rho,$ respectively. 
\end{proposition}
\begin{proof}
We need to analyze equation~\eqref{Bizu1} in two cases. For $i \neq j$, it is rewritten as
\begin{align}\label{Holds-rad}
(m - 2) \frac{F_{x_ix_j}}{F} + f_{x_i x_j} +\frac{F_{x_j}}{F} f_{x_i} +\frac{F_{x_i}}{F}f_{x_j} - \alpha\gamma_{\theta\beta} \phi^{\theta}_{x_i}\phi^{\beta}_{x_j} = 0,
\end{align}
and for $i=j$,
\begin{align}\label{BIZU1-rad}
\nonumber &(m - 2) \frac{F_{x_ix_i}}{F} + \sum_k \frac{ F_{x_k x_k} }{F} - (m - 1)\sum_k\frac{F^2_{x_k}}{F^2} + f_{x_i x_i} +2\frac{F_{x_i}}{F} f_{x_i}-\sum_k\frac{F_{x_k}}{F}f_{x_k}\\
&  - \alpha \gamma_{\theta\beta} \phi^{\theta}_{x_i}\phi^{\beta}_{x_i}  = \frac{\lambda}{F^2}.
\end{align}
Equation~\eqref{Bizu2} is rewritten by means of~\eqref{tensor:field:Euclidean} as
\begin{align}\label{add:Naza-rad}
&\Bigg( F^2 \Big(\sum_k \phi^\theta_{x_kx_k} + (2 - m) \frac{1}{F} \sum_kF_{x_k}\phi^\theta_{x_k}\Big)+ F^2\Big(-\delta_{\beta \theta} \frac{G_{y_\alpha}}{G} - \delta_{\theta\alpha} \frac{G_{y_\beta}}{G} \nonumber\\
&+ \delta_{\alpha\beta} \frac{G_{y_\theta}}{G}\Big) \circ \phi\sum_k\phi_{x_k}^{\alpha}\phi^{\beta}_{x_k}\Bigg)\partial_\theta|_{\phi}= F^2 \sum_kf_{x_k}\phi^\theta_{x_k}\partial_\theta|_{\phi}. 
\end{align}
For any radial smooth function $h(r)$ on $\mathbb{R}^m$, we have $h_{x_i}= h'x_i/r$ and $h_{x_ix_j} =x_ix_j ( h''/r^2 - h'/r^3)$, for all $i \neq j$. Besides, $h_{x_i}= h'x_i/r$ and $h_{x_ix_i} =  x_i^2(h''/r^2 - h'/r^3) +  h'/r ,$  for all $i.$ Thus, from \eqref{Holds-rad} and~\eqref{BIZU1-rad}, we obtain
\begin{align}\label{GuerraeSUA}
x_ix_j\Big[\frac{m - 2 }{F}\Big(\frac{F''}{r^2} - \frac{F'}{r^3}\Big) + \frac{f''}{r^2} - \frac{f'}{r^3} +2\frac{F'}{F} \frac{f'}{r^2}- \alpha \frac{\gamma_{\theta\beta}}{r^2} (\phi^\theta)'(\phi^\beta)'\Big] = 0,
\end{align}
 for all $i\neq j$, and
\begin{align}\label{GuerraeSUA1}
&x_i^2\Big[ \frac{m - 2} {F}\Big(\frac{F''}{r^2} - \frac{F'}{r^3}\Big)  + \frac{f''}{r^2} - \frac{f'}{r^3} +2\frac{F'}{F} \frac{f'}{r^2} - \alpha \frac{\gamma_{\theta\beta}}{r^2} (\phi^\theta)'(\phi^\beta)'  \Big] +  \frac{(m - 2)F'} {rF}\nonumber
\\
&+ \frac{f'}{r} +\frac{F''}{F} 
+ \frac{(m - 1)F'}{F r}    - (m - 1)\Big(\frac{F'}{F}\Big)^2 -\frac{F'}{F}f'  = \frac{\lambda}{F^2},
\end{align}
for all $i$.  The first part of the proposition follows from \eqref{GuerraeSUA} and~\eqref{GuerraeSUA1}. The second one is a straightforward computation from \eqref{add:Naza-rad}.
\end{proof}
\begin{remark}
For constructing a family of mean curvature solitons for MCF in the corresponding self-similar solution to the $(g(t),\phi(t))-(RH)_{\alpha}$ flow background on $M$, one can to use Proposition~\ref{Airtonvaivai} and to consider an $f$-minimal hypersurface $\Sigma$ of $M$ (although we know that hard work is needed to find $f$-minimal hypersurfaces for this case), and then to proceed as in Theorem~\ref{NazasCharacterization} to obtain such a family.    
\end{remark}

\section{\bf Acknowledgements}
The authors would like to express their sincere thanks to the Institute of Mathematics at the Carl von Ossietzky Universität Oldenburg, where part of this work was carried out. The first and second authors are grateful to Prof. Dr. Boris Vertman for their warm hospitality.

\end{document}